\pdfoutput=1
\documentclass[a4paper,fleqn]{cas-sc}

\usepackage[numbers,sort&compress]{natbib}
\usepackage{amsthm,mathtools}
\usepackage[inline]{enumitem}
\usepackage{graphicx}
\usepackage{placeins}
\usepackage{float}
\newfloat{algorithm}{tbp}{loa}
\floatname{algorithm}{Algorithm}
\usepackage{cleveref}
\newcommand{\R}{\mathbb{R}}

\newcommand{\V}{V}
\newcommand{\Vstar}{V^\ast}
\newcommand{\W}{W}
\newcommand{\NNc}{\mathcal{N}}
\newcommand{\dd}{\,\mathrm{d}}
\newcommand{\divx}{\operatorname{div}_x}
\newcommand{\divtx}{\operatorname{div}_{t,x}}
\newcommand{\gradx}{\nabla_x}
\newcommand{\pos}[1]{\left(#1\right)^{+}}
\newcommand{\norm}[2]{\left\|#1\right\|_{#2}}
\newcommand{\pair}[2]{\left\langle #1,#2\right\rangle}
\newcommand{\LQ}{L^2(Q)}
\newcommand{\LO}{L^2(\Omega)}
\newcommand{\bu}{\mathbf{u}}
\newcommand{\bv}{\mathbf{v}}
\newcommand{\uTh}{\bu_{\Theta}}
\newcommand{\Res}{\mathcal{R}}
\newcommand{\enorm}[1]{\left|\!\left|\!\left|#1\right|\!\right|\!\right|_{U}}
\newcommand{\wnorm}[1]{\left|#1\right|_{U}}
\newcommand{\Ainv}{A^{-1/2}}
\newcommand{\Ahalf}{A^{1/2}}

\DeclareMathOperator{\relu}{ReLU}
\DeclareMathOperator{\diam}{diam}
\DeclareMathOperator{\esssup}{ess\,sup}

\crefname{assumption}{Assumption}{Assumptions}
\Crefname{assumption}{Assumption}{Assumptions}
\crefname{remark}{Remark}{Remarks}
\Crefname{remark}{Remark}{Remarks}
\crefname{definition}{Definition}{Definitions}
\Crefname{definition}{Definition}{Definitions}
\crefformat{appendix}{Appendix~#2#1#3}
\Crefformat{appendix}{Appendix~#2#1#3}
\crefmultiformat{appendix}{Appendices~#2#1#3}{ and~#2#1#3}{, #2#1#3}{ and~#2#1#3}
\Crefmultiformat{appendix}{Appendices~#2#1#3}{ and~#2#1#3}{, #2#1#3}{ and~#2#1#3}
\crefformat{subappendix}{Appendix~#2#1#3}
\Crefformat{subappendix}{Appendix~#2#1#3}
\crefmultiformat{subappendix}{Appendices~#2#1#3}{ and~#2#1#3}{, #2#1#3}{ and~#2#1#3}
\Crefmultiformat{subappendix}{Appendices~#2#1#3}{ and~#2#1#3}{, #2#1#3}{ and~#2#1#3}
\crefname{algorithm}{Algorithm}{Algorithms}
\Crefname{algorithm}{Algorithm}{Algorithms}
\Crefname{appendix}{Appendix}{Appendices}

\theoremstyle{plain}
\newtheorem{theorem}{Theorem}[section]
\newtheorem{proposition}[theorem]{Proposition}
\newtheorem{lemma}[theorem]{Lemma}
\newtheorem{corollary}[theorem]{Corollary}
\theoremstyle{definition}
\newtheorem{definition}[theorem]{Definition}
\newtheorem{remark}[theorem]{Remark}
\newtheorem{assumption}[theorem]{Assumption}

\ExplSyntaxOn
\cs_set:Npn \__first_footerline:
{
  \group_begin:
  \small \sffamily
  D.~Damircheli:~{ \rmfamily \itshape Preprint }
  \group_end:
}
\ExplSyntaxOff

\begin{document}
\let\WriteBookmarks\relax
\def\floatpagepagefraction{1}
\def\textpagefraction{.001}

\shorttitle{Certified neural approximations for parabolic obstacle problems}
\shortauthors{D. Damircheli}

\title[mode=title]{Certified neural approximations for parabolic obstacle problems}

\author[1]{Davood Damircheli}
\cormark[1]
\ead{ddamircheli@lsu.edu}
\credit{Conceptualization, Methodology, Formal analysis, Software, Validation,
Investigation, Data curation, Visualization, Writing -- original draft, Writing --
review \& editing}
\affiliation[1]{organization={Department of Mathematics, Louisiana State University},
            city={Baton Rouge},
            postcode={70803},
            state={LA},
            country={USA}}

\cortext[1]{Corresponding author}

\begin{abstract}
Residual-trained neural surrogates carry no error guarantee, a limitation for
free-boundary problems such as phase change and contact. We derive guaranteed error
bounds for neural approximations of parabolic obstacle problems in first-order
state--flux--multiplier form. The main result bounds the flux error of the network
and the energy error of its feasible projection $\max\{v,g\}$ by the constitutive
residual, the initial mismatch, the complementarity defect, and the PDE residual weighted by $c_F^2/\alpha$, without the
coercivity and norm-equivalence constants of the first-order formulation. It extends the Apushkinskaya--Repin majorant for the heat operator to
general coefficients, free multipliers and a flux bound. A penetration-consistency
result shows that neural triples with arbitrarily small bound exist.
Outward-rounded cellwise Taylor-model enclosures turn the bound into a finite
computation valid on any adaptive partition. The flux and gradient
errors of the projected output are certified below $0.27$ and $0.33$ of the
solution scale for a one-dimensional Stefan problem and, after adaptive
refinement, below one half for all three trained networks on a two-dimensional
expanding disk. On an expanding ellipse the gradient bound stays above one half;
substitution diagnostics trace the floor of the trained networks to the projected
state near the free boundary, and the remaining gap to enclosure overhead. Without a closed-form solution, a validated lower bound
replaces the solution norm; there the enclosure is nearly exact, but the trained
networks have errors above the solution scale and no useful accuracy is certified.
The bounds control a state--flux norm, not the multiplier or the free-boundary
location.
\end{abstract}


\maketitle

\section{Introduction}\label{sec:intro}

Parabolic obstacle problems describe evolutions constrained by a unilateral
condition: phase change with a moving melting front, contact of a membrane with a
rigid obstacle, and, outside mechanics, optimal stopping and American option
pricing \cite{Duvaut1973,KinderlehrerStampacchia,Rodrigues,PeskirShiryaev,Seydel}.
The constraint and the complementary reaction generate a moving contact set and
reduced regularity across the free boundary. Neural PDE solvers
\cite{Raissi2019,EYu2018,Sirignano2018} are increasingly used as flexible
surrogates for such problems, but a small empirical training loss does not
control the error to the exact solution or certify complementarity. Where a
neural surrogate is to replace or accompany a conventional simulation, a
computable bound on the error of the delivered network is needed.

This paper provides such bounds. It is organized around the chain
\[
\begin{gathered}
\text{admissible neural triple}\Longrightarrow\text{continuum certificate}\\
\Longrightarrow\text{finite upper-sum certificate}
\Longrightarrow\text{localized diagnostics}.
\end{gathered}
\]
The starting point is the first-order space--time formulation of
\cite{CDFFK2025}, in which the unknown is a state--flux--multiplier triple and the
reliability argument applies to an arbitrary admissible comparison triple, without
Galerkin orthogonality or a discrete variational inequality. We use this
observation as a certification principle for neural approximations.

\paragraph{Main result.} A neural state need not satisfy the obstacle constraint,
but its projection $\widehat v=\max\{v,g\}$ does, and a feasible state can be
inserted directly into the variational inequality. \Cref{thm:direct} exploits
this: for the projected triple $(\widehat v,\tau,\mu)$ and every $\theta\in(0,1]$,
\[
\begin{aligned}
&\norm{\Ainv(\sigma-\tau)}{\LQ}^2+(1-\theta)\norm{\Ahalf\gradx(u-\widehat v)}{\LQ}^2
+\norm{(u-\widehat v)(T)}{\LO}^2\\
&\qquad\le\norm{\widehat S}{\LQ}^2+\norm{u_0-\widehat v(0)}{\LO}^2+2\rho_c^2
+\frac{c_F^2}{\alpha\theta}\norm{\widehat\Res}{\LQ}^2,
\end{aligned}
\]
where $\widehat S$ and $\widehat\Res$ are the constitutive and PDE residuals of the
projected triple and $\rho_c^2$ is the complementarity defect. The estimate is a
functional a posteriori estimate in the sense of Repin \cite{Repin2008} and of
hypercircle type in the sense of Prager and Synge \cite{PragerSynge1947}; it needs
no coercivity lemma, no tunable weight and no norm-equivalence constant. For the heat operator, Apushkinskaya and Repin
\cite{ApushkinskayaRepin2022} already bound the error of any feasible state by a majorant
in which the coincidence set of the approximation enters through a positive part.
\Cref{thm:direct} adds general coefficients $A,b,c$, an explicit and freely chosen
multiplier whose only cost is the complementarity defect, and a bound for the
flux error; its gradient form (\Cref{cor:grad-AR}) contains the heat-operator
majorant (\Cref{rem:AR}). All reported certificates are evaluations of this
theorem.

\paragraph{Attainability.} A reliable bound that cannot be made small on the
neural class would be useless. A penetration-consistency lemma
(\Cref{lem:penetration}) and a density theorem (\Cref{thm:density}) show that the
right-hand side of the direct certificate tends to zero along suitable neural
sequences. The result is qualitative (\Cref{rem:rates}).

\paragraph{Finite certification.} The continuum bound is a sum of integrals. They
are replaced by cellwise upper sums of outward-rounded enclosures on arbitrary
space--time partitions (\Cref{lem:upper-sum}), so any refinement strategy yields a
valid bound and adaptivity affects only cost. The enclosures use standard
interval and Taylor-model arithmetic; the contribution is the per-term evaluation
that matches \Cref{thm:direct}, its treatment of the projection, and a strict
separation of what is proved, what is computed with outward rounding, and what is
diagnostic (\Cref{tab:status}).

\paragraph{Numerical assessment.} Four benchmarks answer four questions. On a
one-dimensional Stefan problem and on an expanding disk in two space dimensions,
the bounds are useful: flux and gradient errors of the projected output are certified below half
the exact solution scale on all seeds. On the disk, refinement removes most of the
enclosure overhead, and the per-term split shows where the remainder sits. On an
expanding ellipse, refinement exposes a floor set by the error of the projected
state near the free boundary, and the gradient bound stays above half the scale
because of enclosure overhead in the residual term. On a problem without
closed-form solution a validated lower bound of the solution scale replaces the
exact norm; the enclosure is nearly exact there, but the trained networks are
inaccurate (\Cref{sec:rising}).

Certification is separated from training: the networks may be trained with any
empirical loss, and the certificate is evaluated independently on the frozen
triple. The certified quantities are the flux error, the energy error of the
projected state and the state--flux weak norm; the multiplier and the
free-boundary location are not certified, and their errors in the manufactured
tests are diagnostics.

\paragraph{Relation to prior work.} A posteriori analysis and space--time
formulations for parabolic variational inequalities include
\cite{Achdou2008,MoonEtAl2007,Dabaghi2020,Fuehrer2020,FuehrerKarkulik2021,GantnerStevenson2021,GantnerStevenson2024,StevensonStorn2023}.
Bounds for an arbitrary conforming approximation without Galerkin orthogonality go
back to the hypercircle method \cite{PragerSynge1947}, developed systematically in
Repin's functional estimates, including the heat equation \cite{Repin2002,Repin2008},
and applied to elliptic obstacle problems in \cite{BraessHoppeSchoeberl2008}. The closest prior result is the parabolic obstacle majorant of
\cite{ApushkinskayaRepin2022}, compared in detail in \Cref{rem:AR}.
Neural PDE solvers and their analysis include
\cite{Raissi2019,EYu2018,Sirignano2018,Han2018,Cai2020,Berrone2022,Lu2021,Wu2023,DeRyck2021,DeRyckMishra2022,Guehring2020,MishraMolinaro2022}.
Closer to certification, Hillebrecht and Unger derive rigorous a posteriori error
bounds for PINNs, and Eiras et al.\ certify worst-case PINN residuals over
continuous domains \cite{HillebrechtUnger2022,Eiras2024}. The present work targets
obstacle problems through an admissible state--flux--multiplier representation,
and bounds an integrated residual--complementarity functional, hence a
solution-error norm, rather than pointwise residuals. Validated integration and range
enclosures are classical
\cite{Moore2009,Tucker2011,FigueiredoStolfi1997,MakinoBerz2003,Johansson2017}.

\paragraph{Outline.} \Cref{sec:model} states the problem and the admissible
approximations, \Cref{sec:direct} the direct certificates, \Cref{sec:neural} the
neural classes, training and attainability, and \Cref{sec:quad} the verified
finite evaluation. \Cref{sec:numerics} reports the numerical assessment and
\Cref{sec:conclusions} concludes. \Cref{app:estimator} contains the
arbitrary-triple reliability route used for comparison, and
\Cref{app:enclosures} the enclosure rules, the arithmetic basis and the reference
norms.
\section{The parabolic obstacle problem and admissible approximations}\label{sec:model}

Let $\Omega\subset\R^d$, $d\ge1$, be a bounded Lipschitz domain, $T>0$, and
$Q=(0,T)\times\Omega$. Formally, we seek $u$ with
\begin{equation}\label{eq:model}
\begin{gathered}
\partial_t u+Lu \ge f,\qquad u\ge g,\qquad (\partial_tu+Lu-f)(u-g)=0\quad \text{in }Q,\\
u=0\quad\text{on }(0,T)\times\partial\Omega,\qquad u(0)=u_0\quad\text{in }\Omega,
\end{gathered}
\end{equation}
where $Lu=-\divx(A\gradx u)+b\cdot\gradx u+cu$. Throughout, $\pair{\cdot}{\cdot}$
denotes the $\Vstar$--$\V$ duality pairing, $\pos{s}=\max\{s,0\}$, and a functional
in $\Vstar$ is nonnegative if $\pair{\mu}{v}\ge0$ for every nonnegative $v\in\V$.

\begin{assumption}[coefficients]\label{ass:coeff}
$A^\top=A\in L^\infty(Q)^{d\times d}$, $b\in L^\infty(Q)^d$ with
$\divx b\in L^\infty(Q)$, $c\in L^\infty(Q)$, $-\tfrac12\divx b+c\ge0$ a.e.\ in $Q$,
and there is $\alpha>0$ with $y\cdot A(t,x)y\ge\alpha|y|^2$ for all $y\in\R^d$, a.e.\
in $Q$. We write $\|A\|_\infty:=\esssup_Q\|A(t,x)\|_2$.
\end{assumption}
The condition on $\divx b$ makes the integration by parts
$(b\cdot\gradx z,z)_{\LQ}=-\tfrac12\int_Q(\divx b)\,z^2$, $z\in\V$, legitimate. All
benchmarks of \Cref{sec:numerics} have $b=0$ and $c=0$.

We use
\[
\V=L^2((0,T);H^1_0(\Omega)),\qquad
\Vstar=L^2((0,T);H^{-1}(\Omega)),
\]
$\W=\V\cap H^1((0,T);H^{-1}(\Omega))$
and $\widetilde W=L^2((0,T);H^1(\Omega))\cap H^1((0,T);H^{-1}(\Omega))$, with
$\norm{v}{\V}=\norm{\gradx v}{\LQ}$. Recall $\W\hookrightarrow C^0([0,T];\LO)$ and,
for $v,w\in\W$,
\begin{equation}\label{eq:ibp}
\begin{aligned}
\int_0^T\pair{\partial_tv(t)}{w(t)}\dd t
={}&-\int_0^T\pair{\partial_tw(t)}{v(t)}\dd t\\
&+(v(T),w(T))_{\LO}-(v(0),w(0))_{\LO}.
\end{aligned}
\end{equation}
Ellipticity gives
\begin{equation}\label{eq:Anorm}
\alpha\norm{v}{\V}^2\le\norm{\Ahalf\gradx v}{\LQ}^2,
\qquad
\norm{\tau}{\LQ}^2\le\|A\|_\infty\norm{\Ainv\tau}{\LQ}^2.
\end{equation}
Let $c_F$ be the Friedrichs constant of $\Omega$, the smallest constant with
$\norm{z}{\LO}\le c_F\norm{\gradx z}{\LO}$ for $z\in H^1_0(\Omega)$. For a box
$\prod_{j=1}^d(a_j,b_j)$, $c_F^2=\bigl(\pi^2\sum_j(b_j-a_j)^{-2}\bigr)^{-1}$, so
$c_F^2=1/\pi^2$ on $(0,1)$ and $c_F^2=1/(2\pi^2)$ on $(0,1)^2$.

\begin{assumption}[data]\label{ass:data}
$f\in\LQ$, $u_0\in\LO$, and $g\in\widetilde W\cap C^0(\overline Q)$ with
\begin{equation}\label{eq:compat}
g|_{(0,T)\times\partial\Omega}\le0,
\qquad g(0)\le u_0\quad\text{in }\Omega.
\end{equation}
Moreover $\psi:=\partial_tg+Lg-f\in\Vstar$ admits a decomposition
$\psi=\psi^+-\psi^-$ with $0\le\psi^\pm\in\Vstar$.
\end{assumption}

\begin{assumption}[obstacle regularity]\label{ass:g-reg}
In addition, $g\in L^2((0,T);H^1(\Omega))\cap H^1((0,T);\LO)$.
\end{assumption}

Let $K_\V=\{v\in\V:v\ge g\}$. The obstacle problem is: find $u\in\W\cap K_\V$ with
$u(0)=u_0$ such that
\begin{equation}\label{eq:VI-canonic}
\pair{\partial_tu+Lu}{v-u}\ge\pair{f}{v-u}
\qquad\forall v\in K_\V.
\end{equation}
Under \Cref{ass:coeff,ass:data}, \eqref{eq:VI-canonic} has a unique solution
\cite[Thm.~1, Rem.~2]{CharrierTroianiello}. With
\begin{equation}\label{eq:def-sigma-lambda}
\sigma:=-A\gradx u\in\LQ^d,
\qquad
\lambda:=\partial_tu+Lu-f\in\Vstar,
\end{equation}
\eqref{eq:VI-canonic} reads $\pair{\lambda}{v-u}\ge0$ for all $v\in K_\V$, and with
$\divtx(v,\tau):=\partial_tv+\divx\tau$ the exact triple $\bu=(u,\sigma,\lambda)$
satisfies the first-order system
\begin{subequations}\label{eq:fos}
\begin{align}
\divtx(u,\sigma)+b\cdot\gradx u+cu-\lambda&=f &&\text{in }Q,\label{eq:fos-a}\\
\sigma+A\gradx u&=0 &&\text{in }Q,\label{eq:fos-b}\\
u\ge g,\quad \lambda\ge0,\quad \lambda(u-g)&=0 &&\text{in }Q,\label{eq:fos-c}
\end{align}
\end{subequations}
with $u=0$ on $(0,T)\times\partial\Omega$ and $u(0)=u_0$; the complementarity line
is understood through \eqref{eq:VI-canonic} whenever the pointwise product is not
defined. The natural graph space and weak norm are
\begin{equation}\label{eq:U}
U:=\left\{(v,\tau,\mu)\in\V\times\LQ^d\times\Vstar:
\divtx(v,\tau)-\mu\in\LQ\right\},
\end{equation}
\begin{equation}\label{eq:wnorm}
\wnorm{\bv}^2
:=\norm{v}{\V}^2+\norm{v(0)}{\LO}^2+\norm{v(T)}{\LO}^2
+\norm{\tau}{\LQ}^2+\norm{\divtx(v,\tau)-\mu}{\LQ}^2 .
\end{equation}
If $(v,\tau,\mu)\in U$, then $v\in\W$, so the traces $v(0)$, $v(T)$ are defined.
The norm controls the combination $\divtx(v,\tau)-\mu$, not $\partial_tv$ and $\mu$
separately, which is the natural topology of the first-order formulation
\cite{CDFFK2025}.

\begin{definition}[admissible triple]\label{def:admissible}
A triple $\bv=(v,\tau,\mu)$ is \emph{admissible} if
\begin{enumerate*}[label=(A\arabic*)]
\item\label{a1} $\bv\in U$;
\item\label{a2} $\mu\in\LQ$ and $\mu\ge0$ a.e.\ in $Q$;
\item\label{a3} $\pos{g-v}\in\W$.
\end{enumerate*}
We write $\mathbb{A}$ for the set of admissible triples.
\end{definition}

Hypothesis \ref{a3} allows the penetration $\pos{g-v}$ to be used as a test
function. The next lemma gives a checkable sufficient condition that covers the
neural classes of \Cref{sec:neural}.

\begin{lemma}[sufficient condition for \ref{a3}]\label{lem:positive-part}
Assume \Cref{ass:g-reg} and $v\in\V\cap H^1((0,T);\LO)$. Then
$w:=\pos{g-v}\in \V\cap H^1((0,T);\LO)\subset\W$, with
\begin{equation}\label{eq:stampacchia}
\gradx w=\mathbf{1}_{\{v<g\}}\gradx(g-v),
\qquad
\partial_t w=\mathbf{1}_{\{v<g\}}\partial_t(g-v)
\qquad\text{a.e.\ in }Q ,
\end{equation}
and $w(0)\le\pos{u_0-v(0)}$ a.e.\ in $\Omega$.
\end{lemma}

\begin{proof}
By \Cref{ass:g-reg}, $g-v\in L^2((0,T);H^1(\Omega))\cap H^1((0,T);\LO)$.
Stampacchia's truncation theorem \cite[Ch.~II]{KinderlehrerStampacchia}, applied for
a.e.\ $t$ in space and, since $z\mapsto z^+$ is $1$-Lipschitz, in time, gives
$w\in L^2((0,T);H^1(\Omega))\cap H^1((0,T);\LO)$ and \eqref{eq:stampacchia}. Since
$v=0$ and $g\le0$ on the lateral boundary, $w(t)\in H^1_0(\Omega)$ for a.e.\ $t$,
so $w\in\V$, and $H^1((0,T);\LO)\subset H^1((0,T);H^{-1}(\Omega))$ gives $w\in\W$.
Both $w$ and $v$ are continuous in time with values in $\LO$, and $g(0)\le u_0$,
so $w(0)=\pos{g(0)-v(0)}\le\pos{u_0-v(0)}$.
\end{proof}
\section{Direct a posteriori certificates for feasible outputs}\label{sec:direct}

\subsection{Feasible projection}
For $\bv=(v,\tau,\mu)\in\mathbb A$ put $w:=\pos{g-v}$, $\widehat v:=\max\{v,g\}=v+w$
and $\widehat{\bv}:=(\widehat v,\tau,\mu)$, and define the complementarity defect
\begin{equation}\label{eq:rho-c}
\rho_c^2(\bv):=\left(\mu,\pos{v-g}\right)_{\LQ}\ \ge 0 .
\end{equation}

\begin{lemma}[feasible projection]\label{lem:projection}
Let \Cref{ass:g-reg} hold and let $\bv\in\mathbb A$ with
$v\in\V\cap H^1((0,T);\LO)$. Then $\widehat v\in\V\cap H^1((0,T);\LO)$,
$\widehat v\ge g$, $\widehat{\bv}\in\mathbb A$, $\pos{\widehat v-g}=\pos{v-g}$, so
$\rho_c(\widehat{\bv})=\rho_c(\bv)$, and $|u_0-\widehat v(0)|\le|u_0-v(0)|$ a.e.\
in $\Omega$.
\end{lemma}
\begin{proof}
\Cref{lem:positive-part} gives $w\in\V\cap H^1((0,T);\LO)$. The graph residual
changes by $\partial_tw\in\LQ$, so $\widehat{\bv}\in U$; $\widehat v\ge g$ and
$\pos{g-\widehat v}=0$ give \ref{a3}. The identity $\pos{\widehat v-g}=\pos{v-g}$ is
pointwise, and $|u_0-\max\{v(0),g(0)\}|\le|u_0-v(0)|$ follows from $g(0)\le u_0$.
\end{proof}

Projection is postprocessing, not a change of the learned parameters. Its state
derivatives are evaluated piecewise with $\mathbf1_{\{v<g\}}$ by
\eqref{eq:stampacchia}; no second derivative of the maximum is needed.

\subsection{The direct certificate}
The projected state is feasible, so it can itself be inserted into the variational
inequality. This removes every penetration term and makes the coercivity lemma of the
first-order formulation
unnecessary: the error is bounded through an identity whose only inequality is
one application of Friedrichs' inequality. If $\widehat\Res=0$, $\rho_c=0$ and
$\widehat v(0)=u_0$, letting $\theta\downarrow0$ in \eqref{eq:direct} below gives
$\norm{\Ainv(\sigma-\tau)}{\LQ}^2+\norm{\Ahalf\gradx(u-\widehat v)}{\LQ}^2
+\norm{(u-\widehat v)(T)}{\LO}^2\le\norm{\widehat S}{\LQ}^2$, a parabolic analogue
of the Prager--Synge hypercircle relation \cite{PragerSynge1947}; with nonzero
residual the estimate is of the majorant type of \cite{Repin2002,Repin2008}.

\begin{theorem}[direct certificate for the feasible output]\label{thm:direct}
Let \Cref{ass:coeff,ass:data,ass:g-reg} hold, let $\bu=(u,\sigma,\lambda)$ be the
exact triple, and let $\bv=(v,\tau,\mu)\in\mathbb A$ with
$v\in\V\cap H^1((0,T);\LO)$. Define the residuals of the projected triple
\begin{equation}\label{eq:direct-residuals}
\widehat\Res:=f-\partial_t\widehat v-\divx\tau-b\cdot\gradx\widehat v-c\widehat v+\mu,
\qquad
\widehat S:=\Ainv\tau+\Ahalf\gradx\widehat v ,
\end{equation}
and write $e:=u-\widehat v$, $e_\sigma:=\sigma-\tau$. Then
\begin{equation}\label{eq:direct-graph}
\divtx(e,e_\sigma)+b\cdot\gradx e+ce-(\lambda-\mu)=\widehat\Res\quad\text{in }\LQ ,
\end{equation}
and for every $\theta\in(0,1]$,
\begin{equation}\label{eq:direct}
\begin{aligned}
&\norm{\Ainv e_\sigma}{\LQ}^2+(1-\theta)\norm{\Ahalf\gradx e}{\LQ}^2
+\norm{e(T)}{\LO}^2\\
&\qquad\le
\norm{\widehat S}{\LQ}^2+\norm{u_0-\widehat v(0)}{\LO}^2+2\rho_c^2(\bv)
+\frac{c_F^2}{\alpha\theta}\norm{\widehat\Res}{\LQ}^2 .
\end{aligned}
\end{equation}
In particular, with $\theta=1$,
\begin{equation}\label{eq:direct-flux}
\norm{\Ainv(\sigma-\tau)}{\LQ}^2
\le\norm{\widehat S}{\LQ}^2+\norm{u_0-\widehat v(0)}{\LO}^2+2\rho_c^2(\bv)
+\tfrac{c_F^2}{\alpha}\norm{\widehat\Res}{\LQ}^2 .
\end{equation}
\end{theorem}

\begin{proof}
By \Cref{lem:projection}, $\widehat{\bv}\in U$, $\widehat v\ge g$ and
$\widehat v\in\V$; hence $e\in\W$ and $e(0)=u_0-\widehat v(0)$. Abbreviate
$F:=\norm{\Ainv e_\sigma}{\LQ}^2$, $G:=\norm{\Ahalf\gradx e}{\LQ}^2$,
$Z_s:=\norm{e(s)}{\LO}^2$.

\emph{Step 1 (constitutive identity).} By \eqref{eq:fos-b},
$e_\sigma+A\gradx e=-(\tau+A\gradx\widehat v)=-\Ahalf\widehat S$, hence
$\Ainv e_\sigma+\Ahalf\gradx e=-\widehat S$ and, by symmetry of $A$,
\begin{equation}\label{eq:direct-step1}
F+G+2\left(e_\sigma,\gradx e\right)_{\LQ}=\norm{\widehat S}{\LQ}^2 .
\end{equation}

\emph{Step 2 (graph identity).} Subtracting the definition of $\widehat\Res$ from \eqref{eq:fos-a} gives
\eqref{eq:direct-graph}. The terms $\partial_te$, $\divx e_\sigma$ and $\lambda-\mu$
belong only to $\Vstar$, but they enter through the grouped expressions
$\divtx(u,\sigma)-\lambda$ and $\divtx(\widehat v,\tau)-\mu$, which lie in $\LQ$
because $\bu,\widehat{\bv}\in U$; hence both sides of \eqref{eq:direct-graph} are in
$\LQ$. Pair \eqref{eq:direct-graph} with $e\in\V$, using the $\Vstar$--$\V$ duality
for each term. Since $e$
vanishes on the lateral boundary, $\pair{\divx e_\sigma}{e}=-(e_\sigma,\gradx e)_{\LQ}$;
by \eqref{eq:ibp}, $\pair{\partial_te}{e}=\tfrac12Z_T-\tfrac12Z_0$; and
$(b\cdot\gradx e+ce,e)_{\LQ}=\int_Q(c-\tfrac12\divx b)\,e^2\ge0$ by
\Cref{ass:coeff}. Therefore
\begin{equation}\label{eq:direct-step2}
-(e_\sigma,\gradx e)_{\LQ}
=(\widehat\Res,e)_{\LQ}-\tfrac12Z_T+\tfrac12Z_0
-\int_Q(c-\tfrac12\divx b)e^2+\pair{\lambda-\mu}{e}.
\end{equation}

\emph{Step 3 (complementarity).} Since $\widehat v\in K_\V$ is admissible in
\eqref{eq:VI-canonic}, $\pair{\lambda}{u-\widehat v}\le0$. Since $\mu\ge0$ and
$u\ge g$, $(\mu,\widehat v-u)_{\LQ}\le(\mu,\widehat v-g)_{\LQ}=\rho_c^2(\bv)$, using
$\widehat v-g=\pos{v-g}$. Hence $\pair{\lambda-\mu}{e}\le\rho_c^2(\bv)$.

\emph{Step 4 (assembly).} Insert \eqref{eq:direct-step2} into
\eqref{eq:direct-step1}, drop the nonnegative convection--reaction term and use
Step~3:
$F+G+Z_T\le\norm{\widehat S}{\LQ}^2+Z_0+2\rho_c^2(\bv)+2(\widehat\Res,e)_{\LQ}$.
Friedrichs' inequality for a.e.\ $t$, \eqref{eq:Anorm} and Young's inequality give
$2(\widehat\Res,e)_{\LQ}\le2c_F\alpha^{-1/2}\norm{\widehat\Res}{\LQ}G^{1/2}
\le\frac{c_F^2}{\alpha\theta}\norm{\widehat\Res}{\LQ}^2+\theta G$, and
\eqref{eq:direct} follows.
\end{proof}

\subsection{Gradient and weak-norm forms}
Write
\[
\mathcal D:=\norm{\widehat S}{\LQ}^2+\norm{u_0-\widehat v(0)}{\LO}^2+2\rho_c^2(\bv),
\qquad
\mathcal L:=\frac{c_F^2}{\alpha}\norm{\widehat\Res}{\LQ}^2,
\]
so that \eqref{eq:direct-flux} reads $\norm{\Ainv(\sigma-\tau)}{\LQ}^2\le\mathcal D+\mathcal L$.

\begin{corollary}[optimal member of the family \eqref{eq:direct}]\label{cor:grad-sharp}
Under the hypotheses of \Cref{thm:direct},
\begin{equation}\label{eq:grad-sharp}
\norm{\Ahalf\gradx(u-\widehat v)}{\LQ}\ \le\ \sqrt{\mathcal L}+\sqrt{\mathcal D+\mathcal L}
=\inf_{\theta\in(0,1)}\Bigl(\frac{\mathcal D+\mathcal L/\theta}{1-\theta}\Bigr)^{1/2}.
\end{equation}
\end{corollary}
\begin{proof}
Step~4 gives, before Young's inequality, $G\le\mathcal D+2\sqrt{\mathcal L}\,G^{1/2}$,
i.e.\ $(G^{1/2}-\sqrt{\mathcal L})^2\le\mathcal D+\mathcal L$. If $\mathcal L>0$, the
function $\theta\mapsto(\mathcal D+\mathcal L/\theta)/(1-\theta)$ attains its minimum
$(\sqrt{\mathcal L}+\sqrt{\mathcal D+\mathcal L})^2$ at
$\theta^\ast=\sqrt{\mathcal L}/(\sqrt{\mathcal L}+\sqrt{\mathcal D+\mathcal L})$; if
$\mathcal L=0$ the infimum is approached as $\theta\downarrow0$.
\end{proof}

\Cref{cor:grad-sharp} is optimal only within \eqref{eq:direct}, whose left-hand
side also contains the flux error. If only the gradient is wanted, the
constitutive relation can be used differently, which removes
$\norm{\widehat S}{\LQ}^2$ from under the second square root.

\begin{corollary}[gradient certificate from the energy pairing]\label{cor:grad-AR}
Under the hypotheses of \Cref{thm:direct}, let $s:=\norm{\widehat S}{\LQ}$,
$\ell:=\sqrt{\mathcal L}$ and
$\mathcal D_0:=\norm{u_0-\widehat v(0)}{\LO}^2+2\rho_c^2(\bv)$. Then
\begin{equation}\label{eq:energy-AR}
\norm{\Ahalf\gradx e}{\LQ}^2+\tfrac12\norm{e(T)}{\LO}^2
\le(s+\ell)\,\norm{\Ahalf\gradx e}{\LQ}+\tfrac12\mathcal D_0 ,
\end{equation}
and consequently
\begin{equation}\label{eq:grad-AR}
\begin{gathered}
\norm{\Ahalf\gradx(u-\widehat v)}{\LQ}\le\mathcal G
:=\tfrac12\Bigl(s+\ell+\sqrt{(s+\ell)^2+2\mathcal D_0}\Bigr),\\
\norm{\Ahalf\gradx e}{\LQ}^2+\norm{e(T)}{\LO}^2\le\mathcal D_0+(s+\ell)^2 .
\end{gathered}
\end{equation}
Moreover $\mathcal G\le\sqrt{\mathcal L}+\sqrt{\mathcal D+\mathcal L}$, with strict
inequality unless $\ell=\mathcal D_0=0$.
\end{corollary}
\begin{proof}
By \eqref{eq:fos-b}, $e_\sigma=-A\gradx e-\Ahalf\widehat S$, so that
$-(e_\sigma,\gradx e)_{\LQ}$ equals $G+(\widehat S,\Ahalf\gradx e)_{\LQ}$. Inserting this into
\eqref{eq:direct-step2}, dropping the convection--reaction term and using Step~3
gives
\[
G+\tfrac12Z_T\le-(\widehat S,\Ahalf\gradx e)_{\LQ}+(\widehat\Res,e)_{\LQ}
+\tfrac12Z_0+\rho_c^2(\bv).
\] The first term is at most $s\,G^{1/2}$, the second at
most $\ell\,G^{1/2}$ as in Step~4; this is \eqref{eq:energy-AR}. Dropping $Z_T$ and
solving the quadratic inequality in $G^{1/2}$ gives the first bound in
\eqref{eq:grad-AR}. Doubling \eqref{eq:energy-AR} gives
$G+Z_T\le2(s+\ell)G^{1/2}-G+\mathcal D_0\le(s+\ell)^2+\mathcal D_0$. The comparison
with \eqref{eq:grad-sharp} is elementary algebra, given in \Cref{app:comparison}.
\end{proof}

\begin{remark}[relation to the majorant of Apushkinskaya and Repin]\label{rem:AR}
Let $A=I$, $b=c=0$, the setting of \cite{ApushkinskayaRepin2022}, and
$\divx\tau\in\LQ$. Put $\widehat\Res_0:=f-\partial_t\widehat v-\divx\tau$ and
\begin{equation}\label{eq:mu-star}
\mu^\ast:=\mathbf 1_{\{v\le g\}}\,\pos{-\widehat\Res_0}.
\end{equation}
Then $(v,\tau,\mu^\ast)$ satisfies the hypotheses of \Cref{thm:direct},
$\rho_c(v,\tau,\mu^\ast)=0$ because $\mu^\ast$ vanishes where $v>g$, and
$\widehat\Res=\widehat\Res_0$ on $\{v>g\}$, $\widehat\Res=\pos{\widehat\Res_0}$ on
$\{v\le g\}=\{\widehat v=g\}$. With $p:=-\tau$, $s=\norm{p-\gradx\widehat v}{\LQ}$
and $\norm{\widehat\Res}{\LQ}$ equals the norm of the residual
$\mathcal F_f(\widehat v,p)$ of \cite[Thm.~2.1]{ApushkinskayaRepin2022}, in which the
residual is replaced by its positive part on the coincidence set of the
approximation. That theorem, with its weight parameter equal to $1$, reads
$\norm{\gradx e}{\LQ}^2+\norm{e(T)}{\LO}^2\le\norm{e(0)}{\LO}^2+(s+\ell)^2$, which is
the second bound of \eqref{eq:grad-AR} with $\rho_c=0$; the first bound,
$\norm{\gradx e}{\LQ}\le\mathcal G$, is never weaker for the gradient alone. In
this sense \Cref{cor:grad-AR} contains the heat-operator majorant;
\Cref{cor:grad-sharp} alone does not (\Cref{app:comparison}).
\Cref{thm:direct,cor:grad-AR} add (i) general coefficients $A,b,c$; (ii) a free
multiplier whose only cost is the computable term $2\rho_c^2$, so that a trained
multiplier with nonzero complementarity defect is certified as it is, while
$\mu^\ast$ remains available as an a~posteriori choice; and (iii) the flux bound
\eqref{eq:direct-flux}, which a state majorant does not provide.
\end{remark}

\begin{corollary}[direct weak-norm certificate]\label{cor:direct-weak}
Under the hypotheses of \Cref{thm:direct} with $b=c=0$, the graph term of
\eqref{eq:wnorm} for $\bu-\widehat{\bv}$ equals $\|\widehat\Res\|_{\LQ}$ by
\eqref{eq:direct-graph}, and for every $\theta\in(0,1)$,
\begin{equation}\label{eq:direct-weak}
\begin{aligned}
\wnorm{\bu-\widehat{\bv}}^2
\le{}&
\max\Bigl\{\frac{1}{\alpha(1-\theta)},\|A\|_\infty,1\Bigr\}
\Bigl(\mathcal D+\frac{\mathcal L}{\theta}\Bigr)\\
&+\norm{u_0-\widehat v(0)}{\LO}^2+\norm{\widehat\Res}{\LQ}^2 .
\end{aligned}
\end{equation}
\end{corollary}
\begin{proof}
By \eqref{eq:Anorm}, $\norm{\gradx e}{\LQ}^2+\norm{e_\sigma}{\LQ}^2+Z_T\le
M\bigl((1-\theta)G+F+Z_T\bigr)$ with $M$ the maximum in \eqref{eq:direct-weak};
apply \eqref{eq:direct} and add $Z_0$ and the graph term.
\end{proof}

\begin{remark}[original output and the estimator route]\label{rem:direct-vs-estimator}
The flux bound \eqref{eq:direct-flux} applies to the network flux as it is. For the
original state, $\gradx(u-v)=\gradx e+\gradx w$ gives
$\norm{\Ahalf\gradx(u-v)}{\LQ}\le\mathcal G+\norm{\Ahalf\gradx w}{\LQ}$; the
penetration also enters $\widehat\Res$ and $\widehat S$ through $\partial_tw$ and
$\gradx w$. \Cref{app:estimator} gives the alternative route of
\cite{CDFFK2025}, with a weight $\Lambda\ge c_F^2$, an energy norm
$\enorm{\cdot}$ and a residual estimator $\rho_r$ defined there: a reliability estimate for arbitrary admissible triples, which
also covers the unprojected output in the weak norm, and the projected bound
$\enorm{\bu-\widehat{\bv}}^2\le6(\rho_r^2(\widehat{\bv})+\rho_c^2(\bv))$
\eqref{eq:projected-direct}. Against the latter, in the flux bound \eqref{eq:direct-flux} the constants on
$\|\widehat S\|^2$ and the initial mismatch drop from $6$ to $1$, that on
$\rho_c^2$ from $6$ to $2$, and the weight of $\|\widehat\Res\|^2$ from
$6\Lambda/\alpha$ to $c_F^2/\alpha$;
with $\Lambda=\diam(\Omega)^2$ this is a factor $6\pi^2\approx59$ on the interval and
$24\pi^2\approx237$ on the unit square. The estimator route is used below only for
comparison (\Cref{tab:disk-same-partition}).
\end{remark}
\section{Neural approximation, training and attainability}\label{sec:neural}

\subsection{Neural triples and admissibility}
A $\tanh$ network with $L$ layers, widths $n_0=1+d,n_1,\dots,n_L$, weights
$W_k\in\R^{n_k\times n_{k-1}}$ and biases $b_k\in\R^{n_k}$ is the map
$N:\R^{1+d}\to\R^{n_L}$, $z^0=z$, $z^k=\tanh(W_kz^{k-1}+b_k)$ for $1\le k\le L-1$,
$N(z)=W_Lz^{L-1}+b_L$; we write $\NNc(L,\mathbf n)$ for the set of such maps. Every
such $N$ is $C^\infty$, and $N$ and all its derivatives are bounded on
$\overline Q$. A \emph{boundary mask} is a function $B:\overline\Omega\to[0,\infty)$
with $B|_{\partial\Omega}=0$ and either
\begin{enumerate*}[label=(M\arabic*)]
\item\label{m1} $B\in W^{1,\infty}(\Omega)$, sufficient for admissibility and the
continuum certificate, or
\item\label{m2} $B\in W^{2,\infty}(\Omega)$, required in addition for the finite
evaluation of \Cref{sec:quad}.
\end{enumerate*}
We therefore avoid nonsmooth distance masks and use the polynomial box mask
$B(x)=\prod_{j=1}^{d}4(x_j-a_j)(b_j-x_j)/(b_j-a_j)^2$, which satisfies \ref{m2}
and $\max_{\overline\Omega}B=1$.

\begin{definition}[neural triple]\label{def:neural-triple}
Given a mask $B$ satisfying \ref{m1} and networks $N_\theta:\R^{1+d}\to\R$,
$S_\phi:\R^{1+d}\to\R^d$, $M_\psi:\R^{1+d}\to\R$, set
\begin{equation}\label{eq:neural-triple}
u_\theta:=B\,N_\theta,
\qquad
\sigma_\phi:=S_\phi,
\qquad
\lambda_\psi:=\left(M_\psi\right)^2 ,
\qquad
\uTh:=(u_\theta,\sigma_\phi,\lambda_\psi).
\end{equation}
\end{definition}

The flux network is free; it is not $-A\gradx u_\theta$. The squared head enforces
$\lambda_\psi\ge0$ exactly. A $\tanh$ network is analytic and cannot vanish on an
open set without vanishing identically; this does not affect admissibility or
$L^2$ attainability, and $(\relu M_\psi)^2$ provides an exact open zero set if
desired, at the price of $C^{1,1}$ instead of $C^\infty$ regularity.

\begin{proposition}[neural admissibility]\label{prop:admissibility}
Let \Cref{ass:g-reg} hold and let $\uTh$ be as in \Cref{def:neural-triple}. Then
$\uTh\in\mathbb{A}$, $u_\theta\in\V\cap H^1((0,T);\LO)$, and
$u_\theta,\gradx u_\theta,\partial_tu_\theta,\sigma_\phi,\divx\sigma_\phi,
\lambda_\psi\in L^\infty(Q)$. Consequently \Cref{thm:direct} and its corollaries
apply to every neural triple, with arbitrary parameters.
\end{proposition}
\begin{proof}
The networks and their first derivatives are bounded on $\overline Q$, and
$B\in W^{1,\infty}(\Omega)$, so $u_\theta=BN_\theta$,
$\gradx u_\theta=N_\theta\gradx B+B\gradx N_\theta$, $\partial_tu_\theta=B\,\partial_tN_\theta$,
$\sigma_\phi$, $\divx\sigma_\phi$ and $\lambda_\psi$ are in $L^\infty(Q)\subset\LQ$.
For a.e.\ $t$, $u_\theta(t,\cdot)$ is the product of a $W^{1,\infty}$ function
vanishing on $\partial\Omega$ with a smooth function, hence in $H^1_0(\Omega)$; thus
$u_\theta\in\V\cap H^1((0,T);\LO)$, and
$\divtx(u_\theta,\sigma_\phi)-\lambda_\psi\in L^\infty(Q)$ gives \ref{a1}.
\ref{a2} holds since $\lambda_\psi=M_\psi^2\ge0$, and \ref{a3} follows from
\Cref{lem:positive-part}.
\end{proof}

\subsection{Training functional versus certificate}
For a neural triple define
\begin{equation}\label{eq:residuals}
\begin{gathered}
R_{\rm PDE}:=f-\divtx(u_\theta,\sigma_\phi)-b\cdot\gradx u_\theta-cu_\theta+\lambda_\psi,\\
R_\sigma:=\Ainv\sigma_\phi+\Ahalf\gradx u_\theta,
\qquad
R_0:=u_0-u_\theta(0,\cdot),
\end{gathered}
\end{equation}
$R_{\rm obs}:=\pos{g-u_\theta}$ and $R_{\rm comp}:=\lambda_\psi\pos{u_\theta-g}$. The
networks are trained on the sampled surrogate
\begin{equation}\label{eq:loss}
\begin{aligned}
\mathcal L_{\rm train}(\Theta)
={}&\omega_r\widehat{\norm{R_{\rm PDE}}{\LQ}^2}
+\omega_\sigma\widehat{\norm{R_\sigma}{\LQ}^2}
+\omega_0\widehat{\norm{R_0}{\LO}^2}\\
&+\omega_p\widehat{\norm{R_{\rm obs}}{\LQ}^2}
+\omega_c\widehat{\textstyle\int_QR_{\rm comp}},
\end{aligned}
\end{equation}
where $\widehat{\ \cdot\ }$ denotes evaluation on a finite collocation set and
$\omega_\bullet>0$ are weights; only first derivatives of the networks are needed.
Three gaps separate $\mathcal L_{\rm train}$ from an error bound: a small sampled
value is compatible with a large integral unless the quadrature error is
controlled, which \Cref{sec:quad} does; the certificate has specific weights that
no reweighting of \eqref{eq:loss} reproduces, and it acts on the projected residuals
\eqref{eq:direct-residuals}, not on $R_{\rm PDE}$ and $R_\sigma$; and it holds only for
admissible triples, so the boundary condition and the sign of the multiplier are
enforced structurally rather than penalized. Certification never uses the training
points or modifies the parameters.

\subsection{Attainability}
By \Cref{thm:direct,cor:grad-AR}, if the right-hand side of \eqref{eq:direct}
tends to zero along a sequence of neural triples, the flux errors and the gradient
errors of the projected states tend to zero. We show that such sequences exist.
Throughout this subsection we assume the following regularity.

\begin{assumption}[regularity of the exact solution]\label{ass:reg}
$u\in\V\cap H^1(Q)$, $\sigma\in\LQ^d$ with $\divx\sigma\in\LQ$, $\lambda\in\LQ$, and
$g\in H^1(Q)$ in addition to \Cref{ass:data,ass:g-reg}.
\end{assumption}

The residual and complementarity terms are continuous in obvious topologies. The
penetration is not obviously so: $\nabla_{t,x}\pos{g-v}$ is the gradient of $g-v$ on
the $v$-dependent set where the constraint is violated, and there is no reason for
it to be small merely because $v-u$ is small. The mechanism that makes it converge
is that on the contact set $\nabla_{t,x}(u-g)=0$ a.e., so the gradient of the
penetration is controlled by the gradient of the error there, while off the contact
set the violation set has vanishing measure.

\begin{lemma}[penetration consistency]\label{lem:penetration}
Let \Cref{ass:reg} hold, let $v_n\in\V\cap H^1(Q)$ with $v_n\to u$ in $H^1(Q)$, and
$w_n:=\pos{g-v_n}$. Then
$\norm{w_n}{\LQ}+\norm{\nabla_{t,x}w_n}{\LQ}+\norm{w_n(0)}{\LO}+\norm{w_n(T)}{\LO}\to0$.
\end{lemma}

\begin{proof}
Let $\Gamma:=\{u=g\}$, $\Gamma^c:=Q\setminus\Gamma=\{u>g\}$ and $S_n:=\{v_n<g\}$.
\emph{(i)} Since $z\mapsto z^+$ is $1$-Lipschitz and $\pos{g-u}=0$,
$|w_n|\le|u-v_n|$ a.e., which gives the zeroth-order terms, at $t=0,T$ through
$H^1(Q)\cap\V\subset H^1((0,T);\LO)\hookrightarrow C^0([0,T];\LO)$.
\emph{(ii)} Since $u-g\in H^1(Q)$, $\nabla_{t,x}(u-g)=0$ a.e.\ on $\Gamma$
\cite[Lemma~7.7]{GilbargTrudinger}.
\emph{(iii)} On $S_n\cap\Gamma^c$ we have $v_n<g<u$, so
$S_n\cap\Gamma^c\subseteq E_n:=\{(t,x)\in\Gamma^c: |u-v_n|>u-g\}$. For $\varepsilon>0$,
$|E_n|\le|\{0<u-g\le\varepsilon\}|+|\{|u-v_n|>\varepsilon\}|$; the second term
tends to $0$ by Chebyshev's inequality, and the first to
$|\{u-g=0\}\cap\Gamma^c|=0$ as $\varepsilon\downarrow0$. Hence $|E_n|\to0$.
\emph{(iv)} By \eqref{eq:stampacchia},
$\norm{\nabla_{t,x}w_n}{\LQ}^2=\int_{S_n\cap\Gamma}|\nabla_{t,x}(g-v_n)|^2
+\int_{S_n\cap\Gamma^c}|\nabla_{t,x}(g-v_n)|^2=:\mathrm T_1+\mathrm T_2$. On $\Gamma$,
(ii) gives $\nabla_{t,x}(g-v_n)=\nabla_{t,x}(u-v_n)$, so
$\mathrm T_1\le\norm{\nabla_{t,x}(u-v_n)}{\LQ}^2\to0$. By (iii),
$\mathrm T_2\le2\int_{E_n}|\nabla_{t,x}(g-u)|^2+2\norm{\nabla_{t,x}(u-v_n)}{\LQ}^2\to0$,
by absolute continuity of the integral of $|\nabla_{t,x}(g-u)|^2\in L^1(Q)$.
\end{proof}

\begin{lemma}[$L^2$ complementarity]\label{lem:l2-complementarity}
Under \Cref{ass:reg}, $\lambda(u-g)=0$ a.e.\ in $Q$ and $(\lambda,u-g)_{\LQ}=0$.
\end{lemma}
\begin{proof}
For nonnegative $\varphi\in C_c^\infty(Q)$ put $\eta=\min(u-g,\varphi)\in\V$; then
$0\le\eta\le u-g$ and $u-\eta\in K_\V$, so \eqref{eq:VI-canonic} gives
$\int_Q\lambda\eta\le0$, hence $=0$ as $\lambda,\eta\ge0$. A countable family of such
$\varphi$ shows $\lambda=0$ a.e.\ on $\{u>g\}$.
\end{proof}

\begin{lemma}[consistency of the direct certificate]\label{lem:consistency}
Let \Cref{ass:coeff,ass:reg} hold and let $\bv_n=(v_n,\tau_n,\mu_n)\in\mathbb A$ with
$v_n\in\V\cap H^1(Q)$ and
\[
v_n\to u\ \text{in }H^1(Q),
\quad
\tau_n\to\sigma,\ \divx\tau_n\to\divx\sigma\ \text{in }\LQ,
\quad
\mu_n\to\lambda\ \text{in }\LQ .
\]
Then $\widehat S_n$, $\widehat\Res_n$, $u_0-\widehat v_n(0)$ and $\rho_c(\bv_n)$, the
quantities of \Cref{thm:direct} for $\bv_n$, tend to zero. Consequently the
right-hand sides of \eqref{eq:direct-flux}, \eqref{eq:grad-sharp} and, for $b=c=0$,
\eqref{eq:direct-weak} tend to zero: the certificate has no consistency gap.
\end{lemma}
\begin{proof}
Let $w_n=\pos{g-v_n}$, so $u-\widehat v_n=(u-v_n)-w_n$, and $\nabla_{t,x}w_n\to0$ and
$w_n\to0$ in $\LQ$ by \Cref{lem:penetration}. Subtracting \eqref{eq:fos-a},
\[
\widehat\Res_n=\partial_t(u-\widehat v_n)+\divx(\sigma-\tau_n)
+b\cdot\gradx(u-\widehat v_n)+c(u-\widehat v_n)+(\mu_n-\lambda),
\]
and every term tends to zero in $\LQ$. By \eqref{eq:fos-b},
$\widehat S_n=\Ainv(\tau_n-\sigma)+\Ahalf\gradx(\widehat v_n-u)\to0$. By
\Cref{lem:projection},
$\norm{u_0-\widehat v_n(0)}{\LO}\le\norm{(u-v_n)(0)}{\LO}\to0$. Finally, since
$\pos{v_n-g}\le(u-g)+|v_n-u|$, $\mu_n\ge0$ and $(\lambda,u-g)_{\LQ}=0$ by
\Cref{lem:l2-complementarity},
$\rho_c^2(\bv_n)\le(\mu_n-\lambda,u-g)_{\LQ}+\norm{\mu_n}{\LQ}\norm{v_n-u}{\LQ}\to0$.
\end{proof}

Transferring this to neural classes requires a compatibility condition between the
exact solution and the multiplicative mask. It is specific to this architecture,
not a consequence of $u\in\V$, and holds for all manufactured solutions below.

\begin{assumption}[mask compatibility]\label{ass:mask}
There exists $\widehat u\in H^1(Q)$ with $u=B\,\widehat u$ a.e.\ in $Q$.
\end{assumption}

\begin{theorem}[existence of neural triples with small certificate]\label{thm:density}
Let \Cref{ass:coeff,ass:reg,ass:mask} hold and $\sigma\in H^1(Q)^d$. Let $\NNc_m$ be the set of
neural triples \eqref{eq:neural-triple} whose networks have depth and width at most
$m$, and let $\eta(\bv)^2:=\mathcal D+\mathcal L$ be the right-hand side of
\eqref{eq:direct-flux}. Then $\inf_{\bv\in\NNc_m}\eta(\bv)\to0$ as $m\to\infty$, and
there are $\bv_m\in\NNc_m$ whose flux errors and projected gradient errors tend to
zero.
\end{theorem}

\begin{proof}
We construct the three components and invoke \Cref{lem:consistency}.
\emph{State.} Smooth functions on a neighbourhood of $\overline Q$ are dense in
$H^1(Q)$, and for $\tanh$, \cite[Thm.~3]{Hornik1991} gives simultaneous uniform
approximation of a smooth function and its first derivatives on $\overline Q$ by
networks in $\NNc(2,\mathbf n)\subset\NNc_m$ for $m\ge\max\{2,n_1\}$; since the
infimum over $\NNc_m$ is nonincreasing in $m$, a diagonal sequence gives
$N_{\theta_m}\to\widehat u$ in $H^1(Q)$. Since $B\in W^{1,\infty}(\Omega)$,
$\norm{u-BN_{\theta_m}}{H^1(Q)}\le C\norm{B}{W^{1,\infty}}\norm{\widehat u-N_{\theta_m}}{H^1(Q)}\to0$.
\emph{Flux.} The same argument, componentwise, gives $S_{\phi_m}\to\sigma$ in
$H^1(Q)^d$, hence $S_{\phi_m}\to\sigma$ and $\divx S_{\phi_m}\to\divx\sigma$ in $\LQ$.
\emph{Multiplier.} Since $0\le\lambda\in\LQ$, $q:=\sqrt\lambda\in L^4(Q)$. Only $L^4$
density is required, which follows from universal approximation in
$C(\overline Q)$; no Sobolev regularity of $\lambda$ is used, which matters because
$\lambda$ may be discontinuous, as in all benchmarks below. Choose $M_{\psi_m}\to q$
in $L^4(Q)$; then
$\norm{M_{\psi_m}^2-\lambda}{\LQ}\le\norm{M_{\psi_m}-q}{L^4}
(\norm{M_{\psi_m}-q}{L^4}+2\norm{q}{L^4})\to0$.
The triples are admissible by \Cref{prop:admissibility}, and
\Cref{lem:consistency} applies.
\end{proof}

\begin{remark}[scope of \Cref{thm:density}]\label{rem:rates}
The result is a qualitative existence statement under \Cref{ass:mask}. No neural
approximation rate is claimed, in particular because the exact multiplier may be
discontinuous across the free boundary. Nor does the theorem say anything about
trainability: it does not assert that a gradient-based optimiser finds such
triples. The expanding ellipse and the rising obstacle of \Cref{sec:numerics},
whose trained multipliers are numerically zero, are examples in which training did
not find them.
\end{remark}
\section{Verified evaluation and adaptive certification}\label{sec:quad}

\Cref{thm:direct} bounds the error by integrals. This section replaces them by
finite computations whose results are guaranteed upper bounds. It needs more
regularity than the continuum theory.

\begin{assumption}[certification regularity]\label{ass:Q}
\begin{enumerate*}[label=(Q\arabic*)]
\item\label{q1} The coefficients, $f$, $u_0$, $g$ and the first derivatives of $g$
admit computable enclosures on boxes (closed form in all benchmarks), and $g$ is
$W^{2,\infty}$ on each component of the complement of a known finite family
$\Gamma_{\rm data}$ of Lipschitz hypersurfaces ($\Gamma_{\rm data}=\emptyset$ in all
benchmarks);
\item\label{q2} the mask satisfies \ref{m2};
\item\label{q3} the activation is $\tanh$.
\end{enumerate*}
\end{assumption}

\subsection{Per-term densities of the projected triple}
Let $\gamma:=g-u_\theta$, so $w=\gamma^+$ and $\widehat v=u_\theta+\gamma^+$. By
\eqref{eq:stampacchia} the projected residuals are piecewise given by two smooth
branches,
\begin{equation}\label{eq:branches}
\begin{aligned}
&\text{on }\{\gamma<0\}: &&\widehat\Res=R_{\rm PDE},\qquad \widehat S=R_\sigma,\\
&\text{on }\{\gamma>0\}: &&\widehat\Res=f-\partial_tg-\divx\sigma_\phi-b\cdot\gradx g-cg+\lambda_\psi,\\
& &&\widehat S=\Ainv\sigma_\phi+\Ahalf\gradx g,
\end{aligned}
\end{equation}
and on $\{\gamma=0\}$ both branches agree a.e.\ because $\nabla_{t,x}\gamma=0$ a.e.\
there. The four quantities of \Cref{thm:direct} are integrals of nonnegative
densities: $\widehat\Res^2$, $|\widehat S|^2$ and the complementarity density
$\lambda_\psi\gamma^-$, $\gamma^-=\pos{u_\theta-g}$, over $Q$, and
$(u_0-\widehat v(0))^2$ over $\Omega$. A cell on which the enclosure of $\gamma$ has
fixed sign uses the corresponding branch; on a \emph{sign-unresolved} cell the hull
of both branches is enclosed. The data interface case is treated in the same way
with one branch per side.

\subsection{Cellwise upper sums}\label{sec:upper-sum}
The implementation uses box-shaped spatial domains; general domains would require
a validated treatment of cut cells. Let $\mathcal T$ be a finite family of closed
boxes with pairwise disjoint interiors covering $\overline Q$, and $\mathcal T_\Omega$
such a family covering $\overline\Omega$.

\begin{lemma}[cellwise upper sums]\label{lem:upper-sum}
Let $\varphi\ge0$ be bounded and measurable on $Q$, and let $\overline\varphi_K\ge\esssup_K\varphi$
for every $K\in\mathcal T$. Then $\int_Q\varphi\le\sum_{K\in\mathcal T}|K|\,\overline\varphi_K$.
The same holds on $\Omega$ with $\mathcal T_\Omega$.
\end{lemma}
\begin{proof}
Cell boundaries are null sets, so $\int_Q\varphi=\sum_K\int_K\varphi\le\sum_K|K|\overline\varphi_K$.
\end{proof}

Writing $\operatorname{sup2}[a,b]:=\max\{a^2,b^2\}$ for an enclosure $[a,b]$, the
per-cell bounds are $\operatorname{sup2}$ of the enclosure of $\widehat\Res$, the sum
over components of $\operatorname{sup2}$ of the enclosures of $\widehat S$,
$\overline\lambda_K\pos{-\underline\gamma_K}$ for the complementarity density, and
$\operatorname{sup2}$ of the enclosure of $u_0-\widehat v(0)$ on trace cells. This
yields the \emph{per-term upper sums}
\[
\widehat U_R\ge\norm{\widehat\Res}{\LQ}^2,\quad
\widehat U_S\ge\norm{\widehat S}{\LQ}^2,\quad
\widehat U_0\ge\norm{u_0-\widehat v(0)}{\LO}^2,\quad
U_c\ge\rho_c^2 .
\]
The required field enclosures are computed from frozen networks and closed-form
data by first-order Taylor models (TM1) with interval coefficients, propagated
through the networks with the chain rule (\Cref{app:enclosures}). Validity of
\Cref{lem:upper-sum} does not depend on how $\mathcal T$ was generated: uniform,
branch-and-bound or data-dependent refinement all give rigorous bounds as soon as
the per-cell enclosures are valid, so adaptivity changes efficiency but not
validity.

\paragraph{Direct finite bounds.} Let $\overline c_F^{\,2}\ge c_F^2$ be a rational
upper bound of the Friedrichs constant. Since the bounds of \Cref{sec:direct} are
increasing in each of the four quantities, the continuum quantities may be replaced
by the upper sums:
\begin{align}
\widehat F_{\rm dir}^2&:=\widehat U_S+\widehat U_0+2U_c
+\frac{\overline c_F^{\,2}}{\alpha}\widehat U_R
\ \ge\ \norm{\Ainv(\sigma-\tau)}{\LQ}^2,\label{eq:direct-finite}\\
\widehat G_{\rm dir}&:=\sqrt{\frac{\overline c_F^{\,2}}{\alpha}\widehat U_R}
+\widehat F_{\rm dir}
\ \ge\ \norm{\Ahalf\gradx(u-\widehat v)}{\LQ},\label{eq:grad-sharp-finite}\\
\widehat G_{\rm AR}&:=\tfrac12\Bigl(\widehat a+\sqrt{\widehat a^{\,2}
+2(\widehat U_0+2U_c)}\Bigr)\ \ge\ \norm{\Ahalf\gradx(u-\widehat v)}{\LQ},
\label{eq:grad-AR-finite}
\end{align}
where $\widehat a:=\sqrt{\widehat U_S}+\sqrt{\overline c_F^{\,2}\widehat U_R/\alpha}$.
These hold by \eqref{eq:direct-flux}, \Cref{cor:grad-sharp} and \Cref{cor:grad-AR},
and $\widehat G_{\rm AR}\le\widehat G_{\rm dir}$. The weak-norm bound $\widehat E_{\rm dir}$
is \eqref{eq:direct-weak} evaluated on the same four sums for rational
$\theta\in\{1/32,\dots,31/32\}$; the smallest value is reported, each candidate being
valid. If only the combined sum
$\widehat U_r:=\frac\Lambda\alpha\widehat U_R+\widehat U_S+\widehat U_0$ is
available, with the weight $\Lambda>\overline c_F^{\,2}$ of \Cref{app:estimator},
the weaker \emph{unsplit} bounds
\begin{equation}\label{eq:direct-nosplit}
\bigl(\widehat F_{\rm dir}^{\rm ns}\bigr)^2:=\widehat U_r+2U_c,
\qquad
\bigl(\widehat G_{\rm dir}^{\rm ns}\bigr)^2:=\frac{\widehat U_r+2U_c}{1-\overline c_F^{\,2}/\Lambda}
\end{equation}
hold, the second with $\theta=\overline c_F^{\,2}/\Lambda$; they serve as a
comparison.

\begin{proposition}[asymptotic exactness under consistent refinement]\label{prop:upper-exactness}
Suppose the densities are Riemann integrable, the maximum cell diameters of the
interior and trace partitions tend to zero, and the per-cell excess satisfies
$0\le\overline\varphi_K-\esssup_K\varphi\le\varepsilon_h\to0$ uniformly. Then each
upper sum converges to its integral.
\end{proposition}
\begin{proof}
Upper Darboux sums converge to the integral; the excess adds at most
$|Q|\varepsilon_h$ (or $|\Omega|\varepsilon_h$).
\end{proof}
This is a sufficient condition; budget-stopped adaptive runs need not satisfy it,
so a small observed gap between upper sum and integral indicates, but does not
prove, sharpness.

\subsection{Adaptive refinement}\label{sec:adaptive}
Cells are marked by D\"orfler's rule \cite{Doerfler1996} on their contributions to
the quantity being reduced. A marked cell is split dyadically in all $d+1$
space--time directions, and every new cell is evaluated from scratch; no child
enclosure is intersected with its parent's. The upper sum is therefore not
guaranteed to be monotone under refinement, and nothing requires it to be:
\Cref{lem:upper-sum} applies to the partition that is actually returned.
\Cref{alg:cert} summarizes the procedure.

\begin{algorithm}[t]
\centering
\fbox{\begin{minipage}{0.94\linewidth}
\small
\textbf{Input:} frozen networks $(N_\theta,S_\phi,M_\psi)$, closed-form data, initial
partitions $\mathcal T^0$ of $\overline Q$ and $\mathcal T_\Omega$ of $\overline\Omega$,
a target quantity $J\in\{\widehat F_{\rm dir}^2,\widehat G_{\rm dir}^2\}$, D\"orfler
parameter $\vartheta$, cell budget.\\[2pt]
\textbf{Repeat} for $k=0,1,\dots$ until the budget is reached:
\begin{enumerate}[leftmargin=*,nosep]
\item \emph{Enclose.} On every cell $K\in\mathcal T^k$, propagate TM1 enclosures of
$u_\theta,\nabla_{t,x}u_\theta,\sigma_\phi,\divx\sigma_\phi,\lambda_\psi$ and the data
in outward-rounded arithmetic; enclose $\gamma$ and select or combine the branches
\eqref{eq:branches}.
\item \emph{Sum.} Form the per-cell bounds and reduce
$\widehat U_R,\widehat U_S,\widehat U_0,U_c$ with exact cell volumes.
\item \emph{Bound.} Evaluate $\widehat F_{\rm dir}$, $\widehat G_{\rm dir}$,
$\widehat G_{\rm AR}$, $\widehat E_{\rm dir}$ by
\eqref{eq:direct-finite}--\eqref{eq:grad-AR-finite} and \eqref{eq:direct-weak}.
\item \emph{Mark and refine.} Mark a minimal set of cells carrying a fraction
$\vartheta$ of $J$ and split them dyadically into $2^{d+1}$ children.
\end{enumerate}
\textbf{Output:} the bounds of step 3 from a from-scratch evaluation of the final
partition.
\end{minipage}}
\caption{Adaptive certification of a frozen neural triple.}\label{alg:cert}
\end{algorithm}

\subsection{Relative bounds without a closed-form solution}\label{sec:nlo}
A relative statement needs a lower bound of the solution scale. With $A=I$ and
$N:=\norm{\gradx u}{\LQ}=\norm{\sigma}{\LQ}$, the triangle inequality and
\Cref{thm:direct,cor:grad-AR} give
\begin{equation}\label{eq:nlo}
N\ \ge\ N_{\rm lo}:=\max\Bigl\{\sqrt{\underline L_\tau}-\widehat F_{\rm dir},\
\sqrt{\underline L_{\nabla\widehat v}}-\widehat G_{\rm AR}\Bigr\},
\end{equation}
where $\underline L_\tau\le\norm{\tau}{\LQ}^2$ and
$\underline L_{\nabla\widehat v}\le\norm{\gradx\widehat v}{\LQ}^2$ are lower sums
$\sum_K|K|\inf_K(\cdot)$ of the same enclosures, with the lower endpoint of a squared
interval taken as $0$ if the interval contains $0$, and with both branches on
sign-unresolved cells. If $N_{\rm lo}>0$, then $\widehat F_{\rm dir}/N_{\rm lo}$ and
$\widehat G_{\rm AR}/N_{\rm lo}$ bound the relative errors. The shortfall satisfies
$N-N_{\rm lo}\le2\widehat F_{\rm dir}+\bigl(\norm{\tau}{\LQ}-\sqrt{\underline L_\tau}\bigr)$,
so the denominator is informative only when the certificate is already small.

\subsection{Status of the computed quantities}\label{sec:status}
The validity of a reported upper sum rests on two ingredients that must be kept
apart: the mathematical enclosure rules of \Cref{app:enclosure-rules}, and the
soundness of their finite-precision execution, whose basis and assumptions are
listed primitive by primitive in \Cref{app:soundness}. Exact-solution errors,
integrals of the same densities computed by float64 quadrature, slacks, shares,
effectivities and fitted slopes are diagnostics. \Cref{tab:status} fixes the status
of every reported quantity; the words \emph{certified} and \emph{validated} are used
only for rows (a)--(c).

\begin{table}[t]
\centering
\caption{Status of the reported quantities.}
\label{tab:status}
\small
\begin{tabular}{p{.04\linewidth}p{.36\linewidth}p{.52\linewidth}}
\toprule
& Quantity & Status\\
\midrule
(a) & \Cref{thm:direct}, its corollaries, \Cref{app:estimator} & Analytic inequalities between continuum integrals.\\
(b) & Per-term upper sums and the absolute bounds $\widehat F_{\rm dir},\widehat G_{\rm dir},\widehat G_{\rm AR},\widehat E_{\rm dir}$ & Outward-rounded evaluation of (a) by \Cref{lem:upper-sum}; valid under \Cref{app:enclosure-rules,app:soundness}. Printed as decimal ceilings of full-precision endpoints, recomputed in $192$-bit Arb.\\
(c) & Relative bounds $\widehat F_{\rm dir}/N$, $\widehat G_{\rm dir}/N$, $\widehat G_{\rm AR}/N$; $N_{\rm lo}$ and ratios over it & Numerator from (b); denominator a $192$-bit lower endpoint of the closed-form $N$ (\Cref{app:norms}) or $N_{\rm lo}$ from downward-rounded lower sums, printed as a decimal floor.\\
(d) & Errors $E(v)$, $E(\widehat v)$ (weak norm \eqref{eq:wnorm}), flux and gradient errors, float64 integrals, slacks, shares, effectivities, slopes, substitutions, the finite-difference reference, all field and partition figures & Diagnostics (float64 quadrature, sampling or a reference solve); no certified status.\\
(e) & Implementation checks (\Cref{app:soundness}) & Tests of the implementation; not a proof of enclosure validity.\\
\bottomrule
\end{tabular}
\end{table}
\section{Numerical assessment}\label{sec:numerics}

All experiments use $A=I$, $b=c=0$ and three independently trained networks
(seeds) per configuration, apart from the two-arm intervention of
\Cref{sec:ellipse}. The main quantities are the direct certificates of
the projected output: the flux bound $\widehat F_{\rm dir}$, the gradient bound
$\widehat G_{\rm AR}$ and, where it is the protocol target, $\widehat G_{\rm dir}$,
each divided by the gradient-and-flux scale $N=\norm{\gradx u}{\LQ}=\norm{\sigma}{\LQ}$
of the exact solution. Since $A=I$, $\widehat F_{\rm dir}/N$ and $\widehat G_{\rm AR}/N$
are validated upper bounds of the relative errors
$\norm{\sigma-\tau}{\LQ}/\norm{\sigma}{\LQ}$ and
$\norm{\gradx(u-\widehat v)}{\LQ}/\norm{\gradx u}{\LQ}$; a statement such as ``certified below $0.42N$'' is an upper bound on the relative
error, not an estimate of it. Diagnostic errors are given next to the bounds so that the gap can be read off;
$E(v)$ and $E(\widehat v)$ denote the weak-norm errors \eqref{eq:wnorm} of the original
and projected triples, and quantities without a hat ($U_R$, $U_S$, $F$, $G_{\rm AR}$)
are float64 integrals of the same densities in place of upper sums. The section
answers four questions:
\begin{enumerate*}[label=(\roman*)]
\item can the bounds be useful (\Cref{sec:useful});
\item what do the direct estimate and refinement improve (\Cref{sec:refinement});
\item what prevents further improvement (\Cref{sec:ellipse}); and
\item does the certificate work without an exact solution (\Cref{sec:rising}).
\end{enumerate*}

\subsection{Implementation and protocols}\label{sec:protocols}
The networks are double-precision $\tanh$ perceptrons with the polynomial mask and
the squared multiplier head; the one-dimensional problems use four hidden layers of
width $64$, the two-dimensional ones six layers of width $64$. The disk baseline
and ellipse runs use $20{,}000$ interior and $2{,}000$ initial collocation points,
$5{,}000$ Adam steps and at most $300$ L-BFGS steps; the long disk schedule uses
$20{,}000$ Adam steps with learning rates $10^{-3}$, $3\cdot10^{-4}$ after step
$8{,}000$ and $10^{-4}$ after step $14{,}000$. In \eqref{eq:loss}, residual,
constitutive and initial weights are one, penetration and complementarity weights
ten. Training and diagnostic evaluation use independent point sets; certification
operates on the frozen trained networks. Wall times refer to a single
workstation (Intel Core Ultra 9 285K) running $16$--$22$ worker processes.

Certification follows \Cref{alg:cert} with exact dyadic cell volumes, a binary64
outward-rounding backend and Arb ball arithmetic \cite{Johansson2017} for
transcendental endpoints and final reductions (\Cref{app:soundness}). All validated
table entries are recomputed in $192$-bit Arb from the full-precision endpoints of
the per-term sums. The one-dimensional problems use uniform partitions. The
two-dimensional problems use two refinement protocols with fixed budgets, marking
rules and targets. \emph{Protocol A}: a
uniform dyadic $128^3$
partition, then D\"orfler marking ($\vartheta=1/2$, with the number of new cells capped by the
budget) on the
contributions to $\widehat F_{\rm dir}^2$ until $3.5\cdot10^6$ cells (one step on the
disk, two on the ellipse). \emph{Protocol B}: from the Protocol A
partition, D\"orfler steps on the contributions to $\widehat G_{\rm dir}^2$, i.e.\ on
$|K|\bigl(\overline S_K+2\overline C_K\bigr)+\bigl(1+\sqrt{(\mathcal D+\mathcal L)/\mathcal L}\bigr)|K|\,\overline c_F^{\,2}\,\overline R_K$,
where $\overline R_K$, $\overline S_K$, $\overline C_K$ are the per-cell bounds of
$\widehat\Res^2$, $|\widehat S|^2$ and $\lambda_\psi\gamma^-$,
with the current global sums, to budgets $4.5$, $6$ and $8\cdot10^6$ cells, with
the target $\widehat G_{\rm dir}/N<1/2$ at the final budget; Protocol B is extended
to seeds $1$ and $2$ only if seed~$0$ meets the target. Both protocols are
formulated in terms of $\widehat G_{\rm dir}$; the sharper bound
$\widehat G_{\rm AR}\le\widehat G_{\rm dir}$ is evaluated on the same sums and
reported alongside. Trace partitions have $16384$ cells, and
the maximum dyadic level is $12$. Two further adaptive dyadic partitions of $10^5$
and $10^6$ cells are used in \Cref{tab:disk-same-partition,tab:disk-budget,tab:ellipse};
by \Cref{lem:upper-sum} the bounds are valid on any partition. The cost per trained network is $1.9$\,h for Protocol A and $4.6$--$4.7$\,h for
Protocol B.

\subsection{Can the bounds be useful?}\label{sec:useful}

\paragraph{One-phase Stefan problem.} Let $\Omega=(0,1)$, $T=1$, $A=1$, $b=c=0$,
$H(t)=e^t-t-1$,
\begin{equation}\label{eq:stefan-data}
f(t,x)=-1-H'(t)(1-x),\qquad u_0=0,\qquad g(t,x)=-H(t)(1-x).
\end{equation}
The exact solution, flux and multiplier are
\begin{equation}\label{eq:stefan-u}
\begin{gathered}
u=-H(t)(1-x)+\mathbf 1_{\{t>x\}}\bigl(e^{t-x}+x-t-1\bigr),\\
\sigma=-H(t)+\mathbf 1_{\{t>x\}}\bigl(e^{t-x}-1\bigr),\qquad
\lambda=\mathbf1_{\{t<x\}},
\end{gathered}
\end{equation}
with free boundary $s(t)=t$. Writing $z=u-g$, the non-contact phase satisfies
$z_t-z_{xx}=-1$; there the temperature $z_t=e^{t-x}-1$ vanishes at
$x=s(t)$ and $s'(t)=-\partial_x z_t(t,s(t))=1$, the one-phase Stefan condition
\cite{Duvaut1973}. The multiplier is discontinuous across the free boundary. The
reference norm is (\Cref{app:norms})
\begin{equation}\label{eq:stefan-reference}
N=\sqrt{23/12-e^2/4}\simeq0.263444 .
\end{equation}

\begin{table}[t]
\centering
\caption{Stefan, uniform $512^2$ partition. Validated: $\widehat F_{\rm dir}$, the
ratios and $\widehat E_{\rm dir}$. Diagnostic: $E(\widehat v)$ and the effectivity
$\widehat E_{\rm dir}/E(\widehat v)$.}
\label{tab:stefan-cert}
\small
\begin{tabular}{rrrrrrrr}
\toprule
Seed & $E(\widehat v)$ & $\widehat F_{\rm dir}$ & $\widehat F_{\rm dir}/N$ & $\widehat G_{\rm dir}/N$ & $\widehat G_{\rm AR}/N$ & $\widehat E_{\rm dir}$ & eff.\\
\midrule
0 & 0.0847 & 0.06847 & 0.260 & 0.330 & 0.319 & 0.1040 & 1.23\\
1 & 0.0866 & 0.06923 & 0.263 & 0.336 & 0.324 & 0.1070 & 1.24\\
2 & 0.0834 & 0.06723 & 0.256 & 0.325 & 0.314 & 0.1031 & 1.24\\
\bottomrule
\end{tabular}
\end{table}

\Cref{tab:stefan-cert} gives the direct certificates on the uniform $512^2$
partition. The flux error of the network and the gradient error of the projected
state are certified to be at most $0.263$ and $0.324$ of the exact scale for every
seed ($0.336$ with $\widehat G_{\rm dir}$), and the weak-norm bound has effectivity $1.23$--$1.24$. The enclosure is
nearly sharp: the per-term upper sums exceed the float64 integrals of the same
densities by factors $1.04$--$1.07$, and the flux certificate lies within a few per
cent of the corresponding exact-integral direct bound ($0.255N$). Against the
diagnostic errors (flux $0.052$, projected gradient $0.037$) the certificates have
effectivities $1.3$ and $2.3$; most of this gap is already present in the
exact-integral bound, i.e.\ in \eqref{eq:direct} rather than in its finite
evaluation. On the same partition the estimator route gives
$\widehat G_U/N=0.80$--$0.83$ and the unsplit direct flux bound $0.33$--$0.34$;
\Cref{fig:stefan-certificates} shows both on coarser uniform partitions. Note that
$E(\widehat v)>E(v)$ here ($0.085$ against $0.072$): feasibility and a tighter
certificate do not imply a smaller error in every norm, because $\partial_tw$ enters
the graph term of the projected residual.

\begin{figure}[t]
\centering
\includegraphics[width=0.62\textwidth,trim=366 0 0 0,clip]{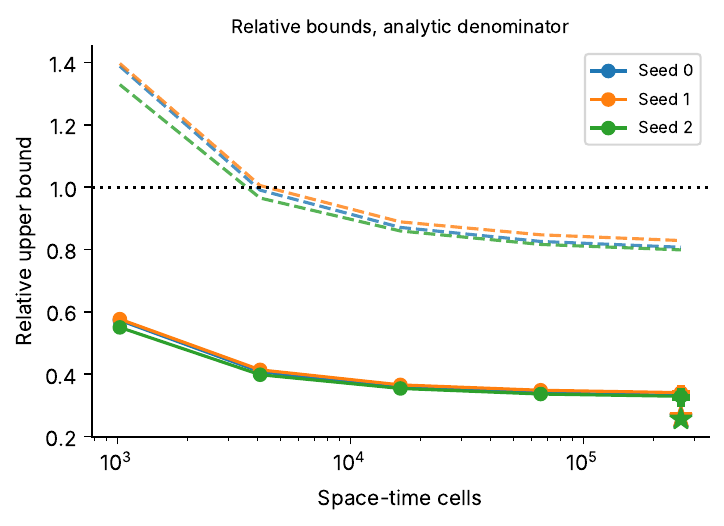}
\caption{Stefan, relative bounds on uniform partitions: unsplit direct flux bound
\eqref{eq:direct-nosplit} (solid), estimator-route projected bound $\widehat G_U$ of
\eqref{eq:projected-finite} (dashed), and the split direct bounds at $512^2$ cells
(stars: $\widehat F_{\rm dir}$; crosses: $\widehat G_{\rm dir}$). Dotted: ratio $1$. All
values validated.}
\label{fig:stefan-certificates}
\end{figure}

\paragraph{Expanding disk.} Let $Q=(0,1)\times\Omega$, $\Omega=(0,1)^2$, $A=I$,
$b=c=0$, $c_0=(\tfrac12,\tfrac12)$, $r=|x-c_0|$, $R(t)=0.15+0.2t$, $H$ as above, and
\[
\zeta=\frac{R(t)^2-r^2}{2R(t)},\qquad
B(x)=16x_1(1-x_1)x_2(1-x_2),\qquad
g=-H(t)B(x).
\]
The manufactured solution is
\begin{equation}\label{eq:disk-exact}
u=g+\kappa_0\zeta_+^2,\qquad
\sigma=-\gradx u,\qquad
\lambda=2\kappa_0\mathbf1_{\{\zeta<0\}},\qquad \kappa_0=5,
\end{equation}
with forcing
$f=g_t-\Delta g-2\kappa_0+\kappa_0\mathbf1_{\{\zeta>0\}}
(2\zeta\zeta_t-2\zeta\Delta\zeta+2-2|\nabla\zeta|^2)$, where
$\zeta_t=R'(R^2+r^2)/(2R^2)$, $\Delta\zeta=-2/R$ and $|\nabla\zeta|^2=r^2/R^2$. Then
\eqref{eq:fos} holds and $u=0$ on $\partial\Omega$. The non-contact set is the
expanding disk $\{r<R(t)\}$, which stays inside the square, and the multiplier jumps
from $0$ to $10$ across $r=R(t)$; at $\zeta=0$ the bracket vanishes, so $f$ is
continuous and belongs to $W^{1,\infty}(Q)$. We set $u_0=u(0,\cdot)$. The reference
norm is closed-form (\Cref{app:norms}), $N\simeq0.688572$.

\begin{table}[t]
\centering
\caption{Long-trained expanding disk under the two refinement protocols of
\Cref{sec:protocols}. Validated: all bounds, ratios and per-term sums.
$\widehat G_{\rm dir}/N$ is the protocol target. Wall time in hours.}
\label{tab:disk-protocol}
\small
\setlength{\tabcolsep}{4pt}
\begin{tabular}{lrrrrrrrrr}
\toprule
Prot. & Seed & $\widehat F_{\rm dir}$ & $\widehat F_{\rm dir}/N$ & $\widehat G_{\rm dir}/N$ & $\widehat G_{\rm AR}/N$ & $\widehat U_R$ & $\widehat U_S$ & $U_c$ & h\\
\midrule
A & 0 & 0.2718 & 0.395 & 0.643 & 0.524 & 0.5735 & 0.02767 & 0.008565 & 1.9\\
A & 1 & 0.3208 & 0.466 & 0.697 & 0.579 & 0.4994 & 0.03391 & 0.02181 & 1.9\\
A & 2 & 0.3188 & 0.463 & 0.731 & 0.596 & 0.6724 & 0.03394 & 0.01678 & 1.9\\
B & 0 & 0.1936 & 0.282 & 0.419 & 0.364 & 0.1777 & 0.01803 & 0.005216 & 4.7\\
B & 1 & 0.2331 & 0.339 & 0.478 & 0.415 & 0.1798 & 0.02246 & 0.01136 & 4.6\\
B & 2 & 0.2263 & 0.329 & 0.482 & 0.414 & 0.2182 & 0.02169 & 0.009223 & 4.6\\
\bottomrule
\end{tabular}
\end{table}

\Cref{tab:disk-protocol} reports the long-trained networks. Under Protocol A
($3.5\cdot10^6$ cells) the flux error is certified below $0.47N$ and the projected
gradient error below $0.60N$ ($\widehat G_{\rm AR}$; $0.74N$ with the target bound
$\widehat G_{\rm dir}$) on all three seeds. Under Protocol B ($8\cdot10^6$ cells) the
target is met on all seeds: flux $0.282$--$0.339$, gradient $0.419$--$0.482$
($\widehat G_{\rm dir}$) and $0.364$--$0.415$ ($\widehat G_{\rm AR}$). Against the
diagnostic errors (flux $0.052$--$0.058$, projected gradient $0.063$--$0.071$) the Protocol B
certificates have effectivities $3.7$--$4.1$ (flux) and $4.6$--$4.7$
($\widehat G_{\rm dir}$); against the exact-integral direct bounds the finite
enclosure costs a factor $2.2$--$2.4$ for the flux ($0.125$--$0.140N$ exact-integral)
and $3.0$--$3.1$ for $\widehat G_{\rm dir}$.

\subsection{What do the direct estimate and refinement improve?}\label{sec:refinement}

\paragraph{Same enclosures, different inequality.} \Cref{tab:disk-same-partition}
evaluates three certificates of the projected output on the same frozen adaptive
partitions of $10^6$ cells, for the baseline and the long-trained disk networks:
the estimator route $\widehat G_U$ (\Cref{app:estimator}), the unsplit direct bounds
\eqref{eq:direct-nosplit}, and the split direct bounds. Only the inequality applied
to the same cellwise enclosures changes. The ratio falls from $7$--$10$ to $3$--$4$
by \Cref{thm:direct}, and to $0.7$--$1.1$ (flux) by replacing the weight $\Lambda=2$
on the residual sum (the estimator weight of \Cref{app:estimator}) by
$\overline c_F^{\,2}=1/(2\pi^2)$. On these partitions the baseline networks are certified below the exact scale for
the flux, but not the long-trained ones.

\begin{table}[t]
\centering
\caption{Expanding disk, frozen adaptive $10^6$-cell partitions: certificates of the
projected output on the same enclosures. $\widehat G_U$: estimator route
\eqref{eq:projected-finite}; superscript ns: unsplit direct bounds
\eqref{eq:direct-nosplit}; last three columns: split direct bounds. Ratios validated;
$E(\widehat v)$ diagnostic.}
\label{tab:disk-same-partition}
\small
\setlength{\tabcolsep}{4pt}
\begin{tabular}{lrrrrrrrr}
\toprule
Schedule & Seed & $E(\widehat v)$ & $\widehat G_U/N$ & $\widehat F_{\rm dir}^{\rm ns}/N$ & $\widehat G_{\rm dir}^{\rm ns}/N$ & $\widehat F_{\rm dir}/N$ & $\widehat G_{\rm dir}/N$ & $\widehat G_{\rm AR}/N$\\
\midrule
Baseline & 0 & 0.2415 & 7.26 & 2.97 & 3.01 & 0.704 & 1.17 & 0.974\\
Baseline & 1 & 0.2677 & 7.29 & 2.98 & 3.02 & 0.719 & 1.19 & 1.01\\
Baseline & 2 & 0.2460 & 8.24 & 3.37 & 3.41 & 0.760 & 1.29 & 1.06\\
Long & 0 & 0.0917 & 9.37 & 3.85 & 3.90 & 1.02 & 1.62 & 1.32\\
Long & 1 & 0.1021 & 9.07 & 3.74 & 3.79 & 1.11 & 1.68 & 1.39\\
Long & 2 & 0.1040 & 10.1 & 4.15 & 4.20 & 1.12 & 1.77 & 1.43\\
\bottomrule
\end{tabular}
\end{table}

\paragraph{Where the slack is.} \Cref{tab:disk-budget} follows seed~$0$ from an
adaptive $10^5$-cell partition to the final Protocol B partition. Three facts organise the
picture. First, the three contributions $\mathcal L$, $\widehat U_S$ and $2U_c$ to
$\widehat F_{\rm dir}^2$ stay within a factor two of each other on the adaptive and
protocol partitions from $10^6$ cells on (on the uniform $128^3$ partition the
residual share is $0.56$); no single term dominates, and $\widehat U_0$ is negligible. Second, the slack of
$\widehat U_R$ and $U_c$ against the float64 integrals of the same densities is two
to three orders of magnitude at $10^6$ cells and still $10^2$ at $8\cdot10^6$,
whereas the slack of $\widehat U_S$ is $2.5$--$5$ from $128^3$ cells on. Sampling $3000$ random cells of the $10^6$ partition indicates that the dominant cause is
the TM1 remainder: the pre-activation radii are small ($\le0.035$), so the $\tanh$
remainders are small, but the per-neuron remainder intervals are mapped through
$|W_k|$ at every linear layer without cancellation and dominate the output radii,
most strongly in the derivative channels (median ratio of TM1 radius to the sampled
variation on the cell: $159$ for $\widehat\Res$, $61$ for $\widehat S$, $26$ for
$\gamma$, $2$ for $\lambda_\psi$); the sup majorant itself costs only $1.2$--$1.8$.
Third, the sign-unresolved region is not a band around the free boundary but the
contact set itself ($82\,\%$ of the volume at $10^6$ cells, $50$--$54\,\%$ at
$8\cdot10^6$), where $v\approx g$ and $|\gamma|$ is below the radius of its
enclosure; there the complementarity density is enclosed by
$\overline\lambda_K\cdot\mathrm{rad}(\gamma)$, so $U_c$ is governed by the state
enclosure rather than by the integral of $\lambda_\psi\gamma^-$. Over the three Protocol B
budgets, a factor $1.8$ in the cell count, the fitted log--log slopes of the squared
bounds against the cell count are $-0.84$ to $-0.91$ (flux) and $-0.95$ to $-1.03$
(gradient) on the three seeds; the constitutive sum, closest to its integral,
decreases more slowly (slopes $-0.57$ to $-0.65$) and carries the largest share at
the end. These are observations over the budgets tested, not asymptotic rates.

\begin{table}[t]
\centering
\caption{Long disk, seed 0: validated direct ratios, shares of
$\widehat F_{\rm dir}^2$ carried by $\mathcal L$, $\widehat U_S$ and $2U_c$, and the
slack of each per-term upper sum against the float64 integral of the same density on
the same partition (diagnostic). The first two rows are adaptive partitions of
$10^5$ and $10^6$ cells; the others follow Protocols A and B.}
\label{tab:disk-budget}
\small
\setlength{\tabcolsep}{3pt}
\begin{tabular}{lrrrrrrrrrr}
\toprule
Partition & Cells & $\frac{\widehat F_{\rm dir}}{N}$ & $\frac{\widehat G_{\rm dir}}{N}$ & $\frac{\widehat G_{\rm AR}}{N}$ & $\mathcal L$ & $\widehat U_S$ & $2U_c$ & \multicolumn{3}{c}{slack $\widehat U_R$, $\widehat U_S$, $U_c$}\\
\midrule
adaptive & 99996 & 6.39 & 11.5 & 8.73 & 0.64 & 0.30 & 0.06 & 139192 & 790.2 & 31736\\
adaptive & 1000000 & 1.02 & 1.62 & 1.32 & 0.35 & 0.34 & 0.31 & 1925 & 22.6 & 3870\\
uniform & 2097152 & 0.514 & 0.897 & 0.687 & 0.56 & 0.28 & 0.16 & 788 & 4.8 & 517\\
A & 3499994 & 0.395 & 0.643 & 0.524 & 0.39 & 0.37 & 0.23 & 330 & 3.8 & 434\\
B & 4500000 & 0.359 & 0.562 & 0.471 & 0.32 & 0.42 & 0.26 & 221 & 3.5 & 400\\
B & 5999995 & 0.318 & 0.483 & 0.411 & 0.27 & 0.43 & 0.30 & 147 & 2.8 & 362\\
B & 8000000 & 0.282 & 0.419 & 0.364 & 0.24 & 0.48 & 0.28 & 102 & 2.5 & 264\\
\bottomrule
\end{tabular}
\end{table}

\paragraph{The networks behind the bounds.} The state error of the seed-$0$ disk
network is of order $10^{-2}$ inside and around the disk (maximum $0.019$, $L^2(Q)$
norm $0.0080$; diagnostic). On the exact contact set the network lies below the
obstacle at more than $99\,\%$ of the grid points of each time slice, so there
$\widehat v=g=u$ up to $2\cdot10^{-5}$; after projection the state error is
confined to the non-contact disk (\Cref{fig:fields}, top) and its $L^2(Q)$ norm
halves to $0.0044$, while the gradient error of $\widehat v$ is larger than that of
$u_\theta$ ($0.063$ against $0.052$). The multiplier network replaces the jump from
$0$ to $10$ by a smooth transition, so $|\lambda-\lambda_\psi|$ is concentrated in a
thin annulus around $r=R(t)$ with maximum $6.7$. On the final Protocol B partition, refinement has added
one dyadic level on $40\,\%$ of the volume, $50\,\%$ of the volume is
sign-unresolved and $79\,\%$ lies in the exact contact set, but only $2.7\,\%$ lies
within one cell size of the exact interface: the refinement effort goes into
enclosure overhead on and near the contact set, not into resolving the interface.

\paragraph{Alternatives to refinement.} Two validated alternatives to TM1 are compared on long-disk seed~$0$; each yields
a complete certificate. Affine
forms (one noise symbol per neuron remainder, reduced to $m$ symbols per cell and
intersected with the TM1 ranges) shrink the per-term sums on a fixed partition by
factors $7$--$140$ on $10^4$ cells, at $7$--$28$ times the TM1 cost per cell; at
equal wall time ($\approx1450$\,s) TM1 refinement reaches $581\,\mathrm{k}$ cells and
$\widehat F_{\rm dir}/N=1.41$, against $30\,\mathrm{k}$ cells and $18.5$ with $m=32$, so
refinement wins by more than an order of magnitude. Second-order Taylor models
(TM2, cubic Lagrange remainders from validated $\tanh^{(3)}$ ranges) cost $3$--$6$
times TM1 per cell; applied to the $145\,\mathrm{k}$ cells with the largest flux
contribution of the Protocol A partition, TM2 lowers $\widehat G_{\rm dir}/N$ from $0.643$ to
$0.569$, a gain $-\Delta\ln(\widehat G_{\rm dir}/N)$ of $0.27$ per $10^6$
TM1-equivalent cell evaluations, against $0.06$--$0.12$ for the three Protocol B refinement
steps and $0.026$ for TM2 on random cells. Targeted TM2 on the cells that carry the
bound is cost-effective in this comparison. A product enclosure of $\lambda_\psi\gamma^-$ had no
effect, since $U_c$ is governed by the radius of $\gamma$ on the contact set.

\subsection{What prevents further improvement?}\label{sec:ellipse}
To test an obstacle that is not proportional to the mask, let
\[
u_b=B(x)(0.10+0.05t)(1+0.2x_1-0.1x_2),\qquad
\zeta=\tfrac{(x_1-0.43)^2}{a(t)^2}+\tfrac{(x_2-0.57)^2}{d(t)^2}-1,
\]
with $a(t)=0.12+0.10t$, $d(t)=0.08+0.07t$, $k_0=0.01$ and $\zeta_-=\max\{-\zeta,0\}$,
and set
\[
\begin{gathered}
g=u_b-k_0\zeta^2,\qquad u=u_b-k_0\zeta_-^2,\qquad \sigma=-\nabla u,\qquad
\lambda=2k_0|\nabla\zeta|^2\mathbf1_{\{\zeta<0\}},\\
f=(\partial_t-\Delta)u_b-2k_0\zeta(\zeta_t-\Delta\zeta)\mathbf1_{\{\zeta<0\}} .
\end{gathered}
\]
Then $u_t-\Delta u-f=\lambda$, $u\ge g$ and complementarity holds; the contact
ellipse stays inside the square, $u=0$ and $g\le0$ on the lateral boundary, and
$u_0=u(0,\cdot)$. The reference norm is closed-form (\Cref{app:norms}),
$N\simeq0.315153$.

\begin{table}[t]
\centering
\caption{Expanding ellipse, direct bounds of the projected output (validated).
Diagnostic weak-norm errors $E(\widehat v)$: $0.1186$, $0.1183$, $0.1238$ for seeds
$0$, $1$, $2$. Protocol B is applied to seed 0 only, since its target is not met there.}
\label{tab:ellipse}
\small
\begin{tabular}{llrrrrr}
\toprule
Seed & Partition & Cells & $\widehat F_{\rm dir}$ & $\widehat F_{\rm dir}/N$ & $\widehat G_{\rm dir}/N$ & $\widehat G_{\rm AR}/N$\\
\midrule
0 & adaptive & 99996 & 2.094 & 6.65 & 11.2 & 9.37\\
 & uniform $128^3$ & 2097152 & 0.2591 & 0.822 & 1.22 & 1.12\\
 & A & 3499994 & 0.1888 & 0.599 & 0.891 & 0.814\\
 & B & 4500000 & 0.1766 & 0.561 & 0.794 & 0.742\\
 & B & 5999995 & 0.1667 & 0.529 & 0.732 & 0.691\\
 & B & 8000000 & 0.1583 & 0.503 & 0.681 & 0.648\\
1 & adaptive & 99996 & 2.651 & 8.41 & 15.2 & 11.8\\
 & uniform $128^3$ & 2097152 & 0.2870 & 0.911 & 1.48 & 1.28\\
 & A & 3499994 & 0.2059 & 0.654 & 1.05 & 0.913\\
2 & adaptive & 99996 & 2.379 & 7.55 & 12.7 & 10.7\\
 & uniform $128^3$ & 2097152 & 0.2731 & 0.867 & 1.34 & 1.20\\
 & A & 3499994 & 0.1992 & 0.632 & 0.956 & 0.866\\
\bottomrule
\end{tabular}
\end{table}

\paragraph{Bounds and their composition.} Under Protocol A
(\Cref{tab:ellipse}) the flux error is certified below $0.66N$ on all three seeds
and the projected gradient error below $0.92N$ ($\widehat G_{\rm AR}$; with
$\widehat G_{\rm dir}$, below $N$ on seeds $0$ and $2$ and $1.05N$ on seed $1$).
Against the diagnostic errors (flux $0.21$--$0.23N$, projected gradient
$0.28$--$0.30N$) the effectivities of $\widehat F_{\rm dir}$ and $\widehat G_{\rm dir}$
are $2.7$--$3.2$ and $3.0$--$3.6$. Under Protocol B, seed~$0$ reaches
$\widehat G_{\rm dir}/N=0.681$ at $8\cdot10^6$ cells ($\widehat G_{\rm AR}/N=0.648$),
above the target, so the protocol is not extended to seeds $1$ and $2$. Replacing every
upper sum by the float64 integral of the same density gives $F/N\approx0.363$ and
$G_{\rm AR}/N\approx0.386$ (diagnostic), against a projected gradient error of
$0.30N$. In the notation of \Cref{cor:grad-AR}, the validated bound at $8\cdot10^6$
cells consists of $\sqrt{\widehat U_S}/N=0.467$ (exact integral $0.362$, slack
$1.67$) and $\sqrt{\overline c_F^{\,2}\widehat U_R}/N=0.178$ (exact integral $0.023$,
slack $58$). The flux certificate is therefore within a factor $1.4$ of its
exact-integral value, and the gradient target is missed because of the residual enclosure: a ratio below $1/2$ would need the slack of $\widehat U_R$ to fall from
$58$ to about $1.7$, which at its fitted slope $-0.93$ means about $44$ times more
cells (a diagnostic extrapolation); the exact-integral value itself is below
$1/2$.

\paragraph{Where the floor comes from.} The multiplier network of seed~$0$ is
numerically zero ($\lambda_\psi\le5\cdot10^{-9}$ on three time slices, against
$\max\lambda=4.5$--$8.4$), and the state network lies below the obstacle on a region
that contains the exact contact ellipse and is $2.15$--$2.32$ times its area
(\Cref{fig:fields}, bottom). Projection replaces the state by $g$ there. On the ring
between the exact contact set and the edge of this region, $\widehat v=g$ while
$u=u_b>g$, and $\gradx\widehat v-\gradx u=-2k_0\zeta\gradx\zeta$ is not small because
the ellipse is small; projection therefore raises the gradient error from $0.052$ to
$0.093$ while it lowers the $L^2(Q)$ error from $0.0100$ to $0.0062$. The ring
carries $95\,\%$ of $U_S$, and $U_S$ carries $99\,\%$ of $F^2$; the refinement
marking concentrates there (\Cref{fig:ellipse-partition}). The substitutions of
\Cref{tab:ellipse-substitution} make this precise. Replacing only the multiplier
(S1) or only the flux (S3) by the exact one makes the bound much larger, because the
trained state and flux satisfy the equation without a multiplier, so that on the
contact set $\widehat\Res\approx\lambda$ once either is corrected. Only the exact flux
and multiplier together (S4) lower it, to $0.304N$, essentially the projected-state
error $0.30N$: with $\tau=\sigma$ one has $\widehat S=\gradx(\widehat v-u)$. The
a~posteriori multiplier $\mu^\ast$ of \eqref{eq:mu-star} halves $U_R$ but lowers
$G_{\rm AR}/N$ only from $0.386$ to $0.378$ (S2). For the trained networks, the floor is thus the error of the projected state near
the free boundary; no choice of flux, multiplier or enclosure
removes it.

\begin{table}[t]
\centering
\caption{Expanding ellipse, seed 0, final Protocol B partition: float64 integrals of the
per-term densities (no enclosure; diagnostic) for the trained projected state with
the trained or substituted flux and multiplier. ``Ring'': the part of $\{v<g\}$
outside the exact contact set. The Protocol A partition gives the same values to
four digits.}
\label{tab:ellipse-substitution}
\footnotesize
\setlength{\tabcolsep}{3pt}
\begin{tabular}{lllrrrrr}
\toprule
Run & Flux & Multiplier & $U_R$ & $U_S$ & $U_S$ on ring & $F/N$ & $G_{\rm AR}/N$\\
\midrule
S0 & $\sigma_\phi$ & $\lambda_\psi$ & $1.08\cdot10^{-3}$ & $1.30\cdot10^{-2}$ & $1.24\cdot10^{-2}$ & 0.363 & 0.386\\
S1 & $\sigma_\phi$ & $\lambda$ (exact) & $4.66\cdot10^{-1}$ & $1.30\cdot10^{-2}$ & $1.24\cdot10^{-2}$ & 0.608 & 0.850\\
S2 & $\sigma_\phi$ & $\mu^\ast$ & $4.83\cdot10^{-4}$ & $1.30\cdot10^{-2}$ & $1.24\cdot10^{-2}$ & 0.363 & 0.378\\
S3 & $\sigma$ (exact) & $\lambda_\psi$ & $4.69\cdot10^{-1}$ & $8.71\cdot10^{-3}$ & $7.04\cdot10^{-3}$ & 0.572 & 0.785\\
S4 & $\sigma$ (exact) & $\lambda$ (exact) & $1.22\cdot10^{-4}$ & $8.71\cdot10^{-3}$ & $7.04\cdot10^{-3}$ & 0.297 & 0.304\\
S5 & $\sigma$ (exact) & $\mu^\ast$ & $3.54\cdot10^{-5}$ & $8.71\cdot10^{-3}$ & $7.04\cdot10^{-3}$ & 0.297 & 0.301\\
\bottomrule
\end{tabular}
\end{table}

\paragraph{A controlled multiplier intervention.} To test whether a nonzero initial
multiplier changes the trained state, two networks are trained with a fourth random seed and the
configuration of seed~$0$, differing only in the output bias of $M_\psi$, which is
$1$ in the intervention arm, so that $\lambda_\psi\approx1$ initially. Four
hypotheses are tested: (P1) the intervention multiplier stays at
least $10^{-3}\max\lambda$, (P2) the area ratio of $\{v<g\}$ to the contact set falls
below $1.5$, and (P3) the exact-integral $U_S$ and (P4) the validated Protocol A bound
$\widehat G_{\rm AR}/N$ are lower than in the control arm. P1 and P2 are rejected: the
intervention multiplier also collapses, from a maximum of $0.86$ to $5\cdot10^{-4}$
within $100$ Adam steps, and the area ratio is $2.18$--$2.33$ in both arms. P3 holds, but only by $0.8\,\%$. P4 holds ($0.817$ against $0.875$), but the exact-integral
$G_{\rm AR}/N$ is $0.398$ in both arms, so the difference lies in enclosure slack. Changing the multiplier initialisation
therefore neither prevented the collapse nor materially changed the projected-state
error; since the collapse occurs in both arms, the experiment does not decide
whether the collapse causes the poor trained state. During training of seed~$0$ the complementarity term of \eqref{eq:loss} falls from
$2.5\cdot10^{-2}$ to $4\cdot10^{-6}$ within $50$ steps, consistent with state and
flux fitting the unconstrained equation early. Lowering the floor requires a more
accurate trained state near the free boundary; a sharper enclosure alone cannot
achieve it.

\begin{figure}[t]
\centering
\includegraphics[width=0.78\textwidth,trim=0 637 0 0,clip]{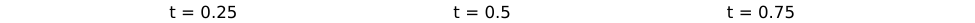}\\[-1pt]
\includegraphics[width=0.78\textwidth,trim=0 328 0 219,clip]{figures/fig_disk_long_solution.pdf}\\[-1pt]
\includegraphics[width=0.78\textwidth,trim=0 3 0 524,clip]{figures/fig_disk_long_solution.pdf}\\[4pt]
\includegraphics[width=0.78\textwidth,trim=0 328 0 219,clip]{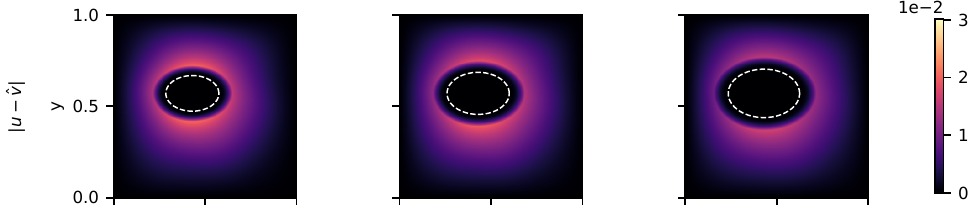}\\[-1pt]
\includegraphics[width=0.78\textwidth,trim=0 125 0 423,clip]{figures/fig_ellipse_solution.pdf}\\[-1pt]
\includegraphics[width=0.78\textwidth,trim=0 3 0 622,clip]{figures/fig_ellipse_solution.pdf}
\caption{Seed-$0$ fields at $t=0.25,0.5,0.75$ on a $401\times401$ grid (diagnostic,
float64). Top two rows, long disk: projected state error $|u-\widehat v|$,
$\widehat v=\max\{u_\theta,g\}$, and multiplier error $|\lambda-\lambda_\psi|$; dashed:
exact free boundary $r=R(t)$. Bottom two rows, ellipse: $|u-\widehat v|$, large on
the ring between the exact contact set (dashed) and the edge of $\{u_\theta<g\}$, and
the trained multiplier $\lambda_\psi$ on the colour scale of the exact $\lambda$,
numerically zero.}
\label{fig:fields}
\end{figure}

\begin{figure}[t]
\centering
\includegraphics[width=0.85\textwidth]{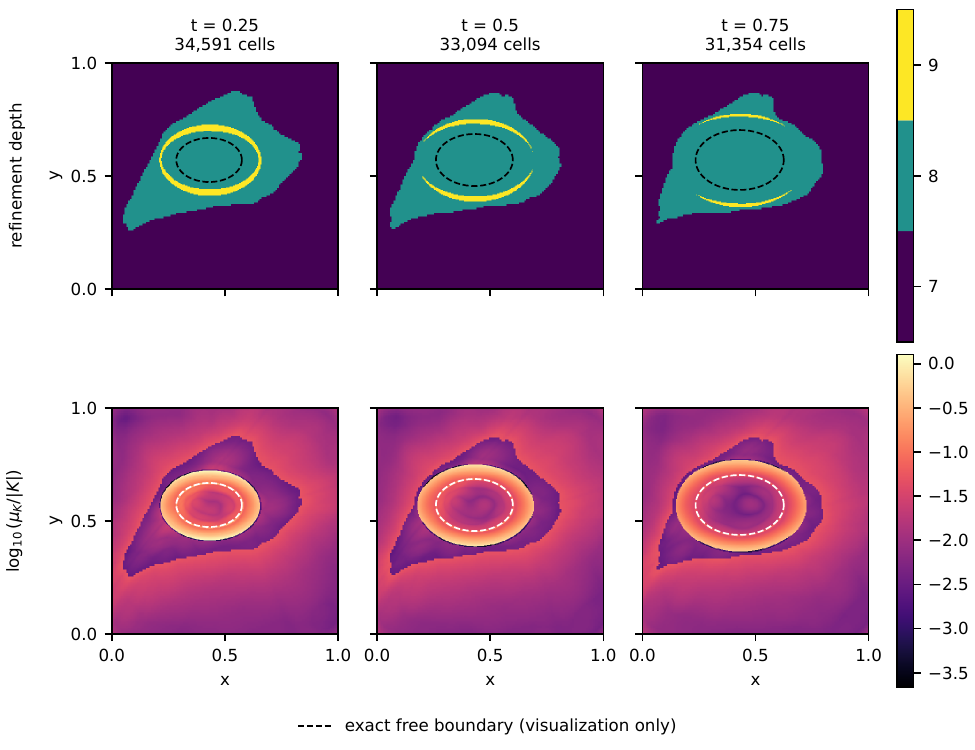}
\caption{Ellipse, seed $0$: the final Protocol B partition ($8\cdot10^6$ space--time cells)
cut at $t=0.25,0.5,0.75$; a cell is drawn if $t_0\le t<t_1$. Top: refinement depth
$-\log_2(\text{cell side})$. Bottom: the Protocol B marking density $\log_{10}(\mu_K/|K|)$.
The marking concentrates on the ring between the exact contact set (dashed; not used
by the partition) and the edge of the region where the projection is active.}
\label{fig:ellipse-partition}
\end{figure}

\subsection{Does the certificate work without an exact solution?}\label{sec:rising}
Let $\Omega=(0,1)$, $T=1$, $A=1$, $b=c=0$ and
\begin{equation}\label{eq:rising-data}
f(t,x)=-8\sin(\pi x),\qquad u_0=0,\qquad g(t,x)=-0.6+0.6\,t-2\bigl(x-\tfrac12\bigr)^2 .
\end{equation}
The forcing pushes the solution onto a rising obstacle; $g\le-0.5$ on the lateral
boundary, $g(0,\cdot)\le u_0$, and $\psi=\partial_tg-\partial_{xx}g-f=4.6+8\sin(\pi x)>0$,
so on the contact set $\lambda=4.6+8\sin(\pi x)$, with a jump of $11.5$--$12.6$
across the free boundary. No closed form of $u$, of the free boundary or of $N$ is
available. As a diagnostic reference we use implicit Euler and lumped $P_1$ elements
with an exact primal--dual active-set solve in every step on $4096^2$ and $2048^2$
grids: the contact set nucleates at $x=1/2$ at $t^\ast\approx0.109$ and widens to
$[0.327,0.673]$ at $t=1$, $N_{\rm ref}=0.86486$ (the two grids differ by
$1.2\cdot10^{-4}$ relative), and the discrete multiplier equals $4.6+8\sin(\pi x)$ on
the contact set to $3\cdot10^{-9}$ away from the free boundary and the nucleation
time.

The networks use the training configuration of the Stefan networks (seeds
$0$--$2$) and are certified on uniform $512^2$ and $1024^2$ partitions with the lower
bound $N_{\rm lo}$ of \eqref{eq:nlo} as denominator; the target is
$\widehat F_{\rm dir}/N_{\rm lo}<1/2$ and $\widehat G_{\rm AR}/N_{\rm lo}<1/2$ at $1024^2$
cells on all seeds. The lower sums are checked where $N$ is known: on the Stefan
certificates of \Cref{tab:stefan-cert}, sampling of $1000$ random cells per seed found no value below
a cell's lower endpoint, and $N_{\rm lo}/N=0.834$, $0.819$, $0.829$; the flux ratios
become $0.312$, $0.321$, $0.308$ instead of $0.260$, $0.263$, $0.256$, and the
gradient ratios $0.382$, $0.396$, $0.379$ instead of $0.319$, $0.324$, $0.314$. This
is the cost of not knowing $N$ when the networks are good.

\begin{table}[t]
\centering
\caption{Rising obstacle (no closed-form solution). Validated: $\widehat F_{\rm dir}$,
$\widehat G_{\rm AR}$ (ceilings), $N_{\rm lo}$ (floor) and the ratios (ceilings).
Diagnostic, against the reference ($N_{\rm ref}=0.86486\pm1.1\cdot10^{-4}$): relative
flux and projected-gradient errors and $N_{\rm lo}/N_{\rm ref}$.}
\label{tab:rising}
\small
\setlength{\tabcolsep}{3.5pt}
\begin{tabular}{lrrrrrrrrrr}
\toprule
 & & \multicolumn{6}{c}{validated} & \multicolumn{3}{c}{diagnostic}\\
\cmidrule(lr){3-8}\cmidrule(lr){9-11}
Cells & Seed & $\widehat F_{\rm dir}$ & $\widehat G_{\rm AR}$ & $N_{\rm lo}$ & $\frac{\widehat F_{\rm dir}}{N_{\rm lo}}$ & $\frac{\widehat G_{\rm dir}}{N_{\rm lo}}$ & $\frac{\widehat G_{\rm AR}}{N_{\rm lo}}$ & flux & grad. & $\frac{N_{\rm lo}}{N_{\rm ref}}$\\
\midrule
$1024^2$ & 0 & 1.445 & 1.629 & 0.211 & 6.85 & 7.78 & 7.72 & 1.247 & 1.051 & 0.244\\
$1024^2$ & 1 & 1.447 & 1.634 & 0.206 & 7.01 & 7.98 & 7.91 & 1.245 & 1.057 & 0.239\\
$1024^2$ & 2 & 1.440 & 1.621 & 0.217 & 6.62 & 7.52 & 7.46 & 1.250 & 1.041 & 0.252\\
$512^2$ & 0 & 1.452 & 1.637 & 0.201 & 7.19 & 8.17 & 8.11 & 1.247 & 1.051 & 0.234\\
$512^2$ & 1 & 1.454 & 1.642 & 0.197 & 7.36 & 8.38 & 8.31 & 1.245 & 1.057 & 0.228\\
$512^2$ & 2 & 1.447 & 1.629 & 0.208 & 6.95 & 7.88 & 7.82 & 1.250 & 1.041 & 0.241\\
\bottomrule
\end{tabular}
\end{table}

\Cref{tab:rising} shows that the target is not met: the ratios are $6.6$--$8.4$
on every seed and budget. The certificate is not the reason. The enclosure is nearly
exact (slacks $1.00$--$1.03$ for every term at $1024^2$ cells, lower sums within
$0.4\,\%$ of their integrals), and against the reference the effectivities are
$1.33$--$1.34$ for the flux and $1.79$--$1.81$ for the gradient bound. The networks are inaccurate: their relative flux error is $1.24$--$1.25$, the projected gradient error
$1.04$--$1.06$, and $E(\widehat v)/N_{\rm ref}\approx1.63$. On every seed the multiplier falls below $10^{-4}$ within $2000$ Adam steps and
ends at most $5.4\cdot10^{-5}$, against the exact $12.6$, while the penetration loss
stays near $0.08$.
The set $\{v\le g\}$ appears at $t=0.110$--$0.111$, at the nucleation time of the
reference, but it is a single interval $2.2$--$3.1$ times as long as the contact set
(\Cref{fig:rising}): the trained state follows the unconstrained equation and
penetrates the obstacle, as on the ellipse but more strongly. Accordingly
$\widehat U_S$ carries $98\,\%$ of $\widehat F_{\rm dir}^2$; the a~posteriori multiplier
$\mu^\ast$ removes the residual term but lowers the exact-integral
$G_{\rm AR}/N_{\rm ref}$ only from $1.87$ to $1.65$. The lower bound $N_{\rm lo}$ is
small, $0.24N_{\rm ref}$, because $\norm{\tau}{\LQ}\approx1.66$ exceeds $N$ and
$\widehat F_{\rm dir}\approx1.44$ is subtracted from it.

Without the reference, the available evidence is a training loss that decreases
by a factor of about $360$--$370$ before stagnating, and validated bounds seven
times the validated lower bound of the solution scale. Without an efficiency
bound, the second observation alone only means that no accuracy is certified; the reference shows that this is not pessimism of the certificate
(effectivities $1.3$--$1.8$) but a property of the networks, whose errors exceed
$N$. Together with \Cref{sec:ellipse}, the example points to a systematic failure of
the penalty-based joint training \eqref{eq:loss}: in both cases the multiplier
network collapses and the state follows the unconstrained equation. For scale, training takes $214$--$235$\,s per seed and certification
$812$--$849$\,s per $1024^2$ partition, whereas the finite-element reference solve
takes about $2$\,s on $4096^2$.

\begin{figure}[t]
\centering
\includegraphics[width=\textwidth]{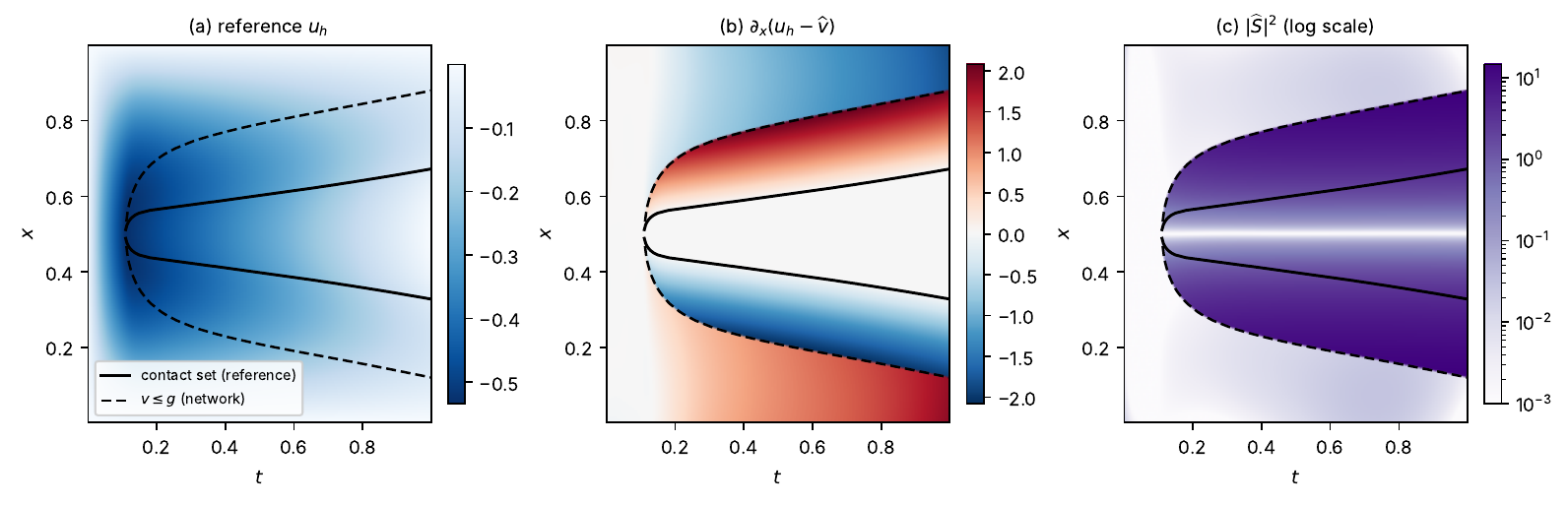}
\caption{Rising obstacle, seed $0$ (diagnostic; reference on $4096^2$). (a) Reference
solution; solid: boundary of the reference contact set; dashed: boundary of
$\{v\le g\}$ for the network. (b) Error of the projected gradient. (c) Constitutive
residual density $|\widehat S|^2$, large on all of $\{v\le g\}$, where
$\widehat v=g$ while the flux follows the unconstrained state.}
\label{fig:rising}
\end{figure}
\section{Conclusions}\label{sec:conclusions}
The direct certificate of \Cref{thm:direct} bounds the flux error of a neural
triple and the energy error of its feasible projection without the coercivity and
norm-equivalence constants of the first-order formulation, and weights the PDE
residual by $c_F^2/\alpha$ only. It extends the heat-operator majorant of Apushkinskaya and Repin to general
coefficients, a freely chosen multiplier and a flux bound, and its gradient form is
at least as sharp as theirs. Penetration consistency shows that the bound can be
made small on neural classes, and cellwise upper sums of outward-rounded
enclosures turn it into a finite, guaranteed computation on arbitrary partitions.

On the Stefan problem and on all three long-trained expanding-disk networks, the
projected output is certified below half the exact scale for the flux and the
gradient. On the disk the finite certificate sits a factor $2.2$--$2.4$ (flux)
above its exact-integral value after refinement; the diagnostics indicate that most
of this slack comes from the first-order Taylor-model remainder propagated through
the linear layers, and in the comparisons of \Cref{sec:refinement} it is reduced most cheaply by
refinement, or by second-order models on the cells that carry the bound. On the
ellipse the exact-integral values are dominated by the constitutive residual on a
ring where the state network lies below the obstacle outside the exact contact set;
substitution tests show that, for the trained networks, this floor is the error of
the projected state near the free boundary, while the remaining
gap of the gradient bound comes from enclosure overhead in the residual term. The
per-term split of \Cref{thm:direct} is what makes this distinction visible. Without
a closed-form solution, a validated lower bound of the solution scale replaces the
reference norm at a cost of about $1.2$ in the ratios when the networks are good; on
the rising obstacle the enclosure is nearly exact, but the networks, whose
multiplier collapsed during training, follow the unconstrained equation, so the bounds exceed the solution scale.

\paragraph{Limitations.} The tests are three manufactured examples and one example
without closed-form solution on box domains in one and two space dimensions.
Uniform refinement requires $O(h^{-(d+1)})$ cells and adaptivity does not remove
worst-case dimensional growth: the two-dimensional certificates below half the
exact scale cost $4.7$ hours per trained network, far more than a conventional
solve. The certificate verifies a given surrogate; no high-dimensional
scalability or advantage over a conventional obstacle solver is established. The
ellipse networks are not converged (stalled L-BFGS), and their gradient
certificate stays above half the exact scale on all seeds. The certificate controls neither the multiplier separately
nor the geometric free boundary, and no efficiency bound is proved. The finite
bounds are as reliable as the enclosure rules and their floating-point execution,
whose basis is listed in \Cref{app:soundness}. Projection guarantees do not
transfer to the original state without the correction of \Cref{rem:direct-vs-estimator}.

\paragraph{Outlook.} The multiplier collapses on the ellipse and the rising obstacle,
and a change of its initialisation alone does not prevent this. Whether the collapse
causes the inaccurate trained states is open, and makes training of the state and multiplier near the free boundary, for instance by
stronger or adaptive enforcement of the constraint, the most important open problem
for the approximations themselves. It is independent of the certification, which
applies unchanged to any admissible triple.

\section*{Data availability}
The source code of the training, certification and reference-solver pipeline,
together with the scripts that regenerate every table and figure (using Arb for the
validated arithmetic), is publicly available on GitHub at
\url{https://github.com/DavoodDamircheli/certified-neural-obstacle} under the MIT
license; the version used in this paper (v1.0.0) is archived on Zenodo
\cite{DamircheliCode2026}. The repository includes the trained Stefan network of
seed~$0$ with its stored certificate. The remaining trained networks, the frozen
interior and trace partitions with their per-cell records, and the full-precision
per-term endpoints of every run are available from the author on request. Every validated table entry can be re-reduced in Arb from the per-cell records,
and every certificate can be recomputed from the trained networks and the
partition; for Stefan seed~$0$ and long-disk seed~$0$ under Protocol A this
recomputation reproduces all recorded quantities bit for bit.

\printcredits

\section*{Declaration of competing interest}
The author declares no known competing financial interests or personal
relationships that could have appeared to influence the work reported in this
paper.

\section*{Funding}
This research did not receive any specific grant from funding agencies in the
public, commercial, or not-for-profit sectors.

\appendix
\renewcommand{\thetheorem}{\Alph{section}.\arabic{theorem}}
\renewcommand{\thetable}{\arabic{table}}
\section{The arbitrary-triple reliability route}\label{app:estimator}
This appendix records the reliability estimate of the first-order framework
\cite{CDFFK2025}, with explicit constants. It covers the unprojected output and supplies the
comparison bounds of \Cref{tab:disk-same-partition} and \Cref{fig:stefan-certificates};
none of the main certificates depends on it.

\subsection{Energy norm and coercivity}
Fix $\Lambda\ge c_F^2$. For $\bv=(v,\tau,\mu)\in U$ let
$D(\bv):=\divtx(v,\tau)+b\cdot\gradx v+cv-\mu$ and
\begin{equation}\label{eq:enorm}
\begin{aligned}
\enorm{\bv}^2
:={}&\norm{\Ahalf\gradx v}{\LQ}^2+\norm{\Ainv\tau}{\LQ}^2\\
&+\norm{v(0)}{\LO}^2+\norm{v(T)}{\LO}^2
+\frac{\Lambda}{\alpha}\norm{D(\bv)}{\LQ}^2 .
\end{aligned}
\end{equation}
By \eqref{eq:Anorm} and Young's inequality, $\wnorm{\bv}^2\le C_{\rm eq}\enorm{\bv}^2$
with
$C_{\rm eq}:=\max\{\alpha^{-1}+4(\norm{b}{L^\infty}^2+\Lambda\norm{c}{L^\infty}^2)/\alpha,
\|A\|_\infty,1,2\alpha/\Lambda\}$; for $b=c=0$ the last entry may be replaced by
$\alpha/\Lambda$, so $C_{\rm eq}=1$ for all benchmarks. With
\[
\begin{aligned}
a(\bu,\bv):={}&\tfrac{\Lambda}{\alpha}\left(D(\bu),D(\bv)\right)_{\LQ}
+\left(u(0),v(0)\right)_{\LO}+\tfrac12\pair{\lambda}{v}+\tfrac12\pair{\mu}{u}\\
&+\left(\Ainv\sigma+\Ahalf\gradx u,\Ainv\tau+\Ahalf\gradx v\right)_{\LQ},
\end{aligned}
\]
\cite[Lemma~4, eq.~(11)]{CDFFK2025} gives
\begin{equation}\label{eq:coercive}
\enorm{\bv}^2\le6\,a(\bv,\bv)\qquad\forall\bv\in U .
\end{equation}

\subsection{Reliability for arbitrary admissible triples}
For $\bv\in\mathbb A$ with $w=\pos{g-v}$ and
$\Res:=f-\divtx(v,\tau)-b\cdot\gradx v-cv+\mu$ define
\begin{align}
\rho_r^2(\bv)&:=\frac{\Lambda}{\alpha}\norm{\Res}{\LQ}^2
+\norm{\Ainv\tau+\Ahalf\gradx v}{\LQ}^2+\norm{u_0-v(0)}{\LO}^2,\label{eq:rho-r}\\
\rho_p^2(\bv)&:=\norm{\Ahalf\gradx w}{\LQ}^2+\kappa\norm{w}{\LQ}^2
+\tfrac{1}{\alpha}\norm{\partial_tw}{\Vstar}^2\notag\\
&\qquad+\norm{w(0)}{\LO}^2+\norm{w(T)}{\LO}^2,
\label{eq:rho-p}
\end{align}
$\kappa:=\alpha/\Lambda+(\norm{b}{L^\infty}^2+\Lambda\norm{c}{L^\infty}^2)/\alpha$, and
$\rho^2:=\rho_r^2+\rho_p^2+\rho_c^2$ with $\rho_c$ from \eqref{eq:rho-c}.

\begin{theorem}[reliability for admissible triples]\label{thm:reliability}
Under \Cref{ass:coeff,ass:data,ass:g-reg}, for every $\bv\in\mathbb A$,
\begin{equation}\label{eq:reliability}
\enorm{\bu-\bv}^2\le18\rho_r^2(\bv)+144\rho_p^2(\bv)+12\rho_c^2(\bv)\le144\,\rho^2(\bv),
\end{equation}
and $\wnorm{\bu-\bv}^2\le144\,C_{\rm eq}\rho^2(\bv)$. If $b=c=0$, the constant $144$ on
$\rho_p^2$ can be replaced by $36$.
\end{theorem}

\begin{proof}
Let $\mathbf e=\bu-\bv$ and $z:=\max\{g,v\}=v+w$. By \eqref{eq:fos}, the first
argument of $D$ in $a(\mathbf e,\mathbf e)$ equals $\Res$, the constitutive term
equals $\norm{\Ainv\tau+\Ahalf\gradx v}{\LQ}^2$ and $(u-v)(0)=u_0-v(0)$, so
\begin{equation}\label{eq:identity}
a(\mathbf e,\mathbf e)=\rho_r^2(\bv)+\pair{\lambda-\mu}{u-v}.
\end{equation}
By \ref{a3}, $w\in\W\subset\V$, so $z\in K_\V$ is admissible in \eqref{eq:VI-canonic},
which gives $\pair{\lambda}{u-z}\le0$, a genuine $\Vstar$--$\V$ pairing; no pairing of
$\lambda$ with $g$ is needed. Hence $\pair{\lambda}{u-v}\le\pair{\lambda}{w}$, and
$\mu\ge0$, $u\ge g$ give $(\mu,v-u)_{\LQ}\le(\mu,v-g)_{\LQ}$. With
$w+(v-g)=\pos{v-g}$,
\begin{equation}\label{eq:step2}
\pair{\lambda-\mu}{u-v}\le\pair{\lambda-\mu}{w}+\rho_c^2(\bv).
\end{equation}
Subtracting the definition of $\Res$ from \eqref{eq:fos-a},
$\lambda-\mu=-\Res+\partial_t(u-v)+\divx(\sigma-\tau)+b\cdot\gradx(u-v)+c(u-v)$ in
$\Vstar$. Let $G,F,Z_s$ be the gradient, flux and trace parts of
$\enorm{\mathbf e}^2$ and $\delta\in(0,1]$. Pairing each term with $w\in\W$, using
\eqref{eq:ibp} for the time derivative, $\pair{\divx\chi}{w}=-(\chi,\gradx w)_{\LQ}$,
Friedrichs' inequality with $c_F\le\Lambda^{1/2}$, Cauchy--Schwarz and Young, with
$\mathcal Cz:=b\cdot\gradx z+cz$ and
$\beta:=(\norm{b}{L^\infty}^2+\Lambda\norm{c}{L^\infty}^2)/\alpha$,
\begin{align*}
\pair{-\Res}{w}&\le\tfrac12\rho_r^2+\tfrac1{2\delta}\tfrac{\alpha}{\Lambda}\norm{w}{\LQ}^2,\\
\pair{\partial_t(u-v)}{w}&\le\tfrac1{2\delta}\bigl(\tfrac1\alpha\norm{\partial_tw}{\Vstar}^2
+\norm{w(0)}{\LO}^2+\norm{w(T)}{\LO}^2\bigr)\\
&\quad+\tfrac\delta2(G+Z_0+Z_T),\\
\pair{\divx(\sigma-\tau)}{w}&\le\tfrac\delta2F+\tfrac1{2\delta}\norm{\Ahalf\gradx w}{\LQ}^2,\\
(\mathcal C(u-v),w)_{\LQ}&\le\tfrac\delta2G+\tfrac1\delta\beta\norm{w}{\LQ}^2,
\end{align*}
The error contributions sum to at most $\delta\enorm{\mathbf e}^2$ and the
penetration contributions to at most $\delta^{-1}\rho_p^2$, so
$\pair{\lambda-\mu}{w}\le\tfrac12\rho_r^2+\delta^{-1}\rho_p^2+\delta\enorm{\mathbf e}^2$.
Combining with \eqref{eq:coercive}, \eqref{eq:identity}, \eqref{eq:step2} and
absorbing for $\delta<1/6$,
\[
\enorm{\mathbf e}^2\le\frac{9\rho_r^2+6\rho_c^2+6\delta^{-1}\rho_p^2}{1-6\delta},
\]
and $\delta=1/12$ gives \eqref{eq:reliability}. If $b=c=0$, the last estimate
vanishes, the error contributions are at most $\tfrac\delta2\enorm{\mathbf e}^2$ and
each penetration coefficient is at most $1/(2\delta)$, so
$(1-3\delta)\enorm{\mathbf e}^2\le9\rho_r^2+6\rho_c^2+3\delta^{-1}\rho_p^2$, and
$\delta=1/6$ gives the constant $36$.
\end{proof}

\subsection{Computable variant and the projected estimator bound}
Since $\norm{\varphi}{\Vstar}\le c_F\norm{\varphi}{\LQ}\le\Lambda^{1/2}\norm{\varphi}{\LQ}$,
replacing $\alpha^{-1}\norm{\partial_tw}{\Vstar}^2$ by
$\frac{\Lambda}{\alpha}\norm{\partial_tw}{\LQ}^2$ in \eqref{eq:rho-p} gives a computable
$\widetilde\rho_p\ge\rho_p$, and \Cref{thm:reliability} holds with $\widetilde\rho$; all
terms of $\widetilde\rho^2$ are integrals of local densities, so it localizes exactly
over any partition and can drive marking; no local reliability estimate is
claimed. For neural triples this gives
$\wnorm{\bu-\uTh}^2\le C_{\rm eq}C_{\rm rel}\widetilde\rho^2(\uTh)$ with $C_{\rm rel}=144$,
or $36$ if $b=c=0$. By the argument of \Cref{lem:consistency} with
\Cref{lem:penetration}, $\widetilde\rho(\bv_n)\to0$ under the same hypotheses.

For the projected triple, $\pair{\lambda}{u-\widehat v}\le0$ and $\pos{g-\widehat v}=0$
remove the penetration terms, and \eqref{eq:identity} with \eqref{eq:coercive} gives
\begin{equation}\label{eq:projected-direct}
\enorm{\bu-\widehat{\bv}}^2\le6\bigl(\rho_r^2(\widehat{\bv})+\rho_c^2(\bv)\bigr).
\end{equation}
Its finite form, with $\widehat U_r$ as in \eqref{eq:direct-nosplit} and
$C_{\rm eq}=1$, is
\begin{equation}\label{eq:projected-finite}
\widehat G_U:=\sqrt{6\bigl(\widehat U_r+U_c\bigr)}
\ \ge\ \max\bigl\{\norm{\Ainv(\sigma-\tau)}{\LQ},\norm{\Ahalf\gradx(u-\widehat v)}{\LQ}\bigr\},
\end{equation}
the estimator-route column of \Cref{tab:disk-same-partition}.

\subsection{Algebra for \Cref{cor:grad-AR} and \Cref{rem:AR}}\label{app:comparison}
Put $a=s+\ell$ and $W=(s^2+\ell^2+\mathcal D_0)^{1/2}$, so that
$\sqrt{\mathcal L}+\sqrt{\mathcal D+\mathcal L}=\ell+W$. The claim $\mathcal G\le\ell+W$
is $\sqrt{a^2+2\mathcal D_0}\le\ell-s+2W$, whose right-hand side is nonnegative. After
squaring, the difference of the two sides is
$4(s^2+\ell^2-s\ell)+2\mathcal D_0-4(s-\ell)W$, nonnegative if $\ell\ge s$. If
$s>\ell$, put $p=s^2+\ell^2-s\ell$; then $(s-\ell)^2=p-s\ell$,
$W^2=p+s\ell+\mathcal D_0$ and
$(p+\tfrac12\mathcal D_0)^2-(s-\ell)^2W^2=\tfrac14\mathcal D_0^2+s^2\ell^2+s\ell\mathcal D_0\ge0$.
Equality throughout forces $\ell=\mathcal D_0=0$. Moreover
$\mathcal G^2=\tfrac12\bigl(a^2+\mathcal D_0+a\sqrt{a^2+2\mathcal D_0}\bigr)\le a^2+\mathcal D_0$.
Finally, \Cref{cor:grad-sharp} does not contain the Apushkinskaya--Repin bound,
since $(\ell+W)^2-\mathcal D_0-(s+\ell)^2=\ell(\ell+2W-2s)>0$ whenever $\ell>0$.
\section{Enclosures, arithmetic and reference norms}\label{app:enclosures}

\subsection{Enclosure rules}\label{app:enclosure-rules}
For an interval $I=[a,b]$, $\operatorname{supabs}(I)=\max\{|a|,|b|\}$ and
$\operatorname{sup2}(I)=\operatorname{supabs}(I)^2$. Addition and multiplication are
the hulls of the endpoint operations; squaring gives $[0,\max(a^2,b^2)]$ if $0\in I$
and $[\min(a^2,b^2),\max(a^2,b^2)]$ otherwise; $\tanh(I)=[\tanh a,\tanh b]$. Let
$Z^{k-1},J^{k-1},H^{k-1}$ enclose a layer value, Jacobian and Hessian on a box. For
$y^k=W_kz^{k-1}+b_k$ enclose $Y^k=W_kZ^{k-1}+b_k$, $Y^k_{,i}=W_kJ^{k-1}_{,i}$,
$Y^k_{,ij}=W_kH^{k-1}_{,ij}$, and
\[
\begin{gathered}
Z^k=\tanh(Y^k),\qquad
J^k_{,i}=\tanh'(Y^k)\odot Y^k_{,i},\\
H^k_{,ij}=\tanh''(Y^k)\odot Y^k_{,i}\odot Y^k_{,j}+\tanh'(Y^k)\odot Y^k_{,ij},
\end{gathered}
\]
with componentwise interval products, $J^0=I$, $H^0=0$ and an affine output layer.
With $P_0(y)=y$ and $P_{q+1}(y)=(1-y^2)P_q'(y)$, $\tanh^{(q)}=P_q(\tanh)$, so
interval evaluation of $P_q$ gives derivative ranges, with $|\tanh'|\le1$ and
$|\tanh''|\le4/(3\sqrt3)$ as global fallbacks. A first-order Taylor model of a scalar map $F$ about an expansion point $a_0$ on an
enclosed range $I=[\underline a,\overline a]\ni a_0$ is
$F(a)\in F(a_0)+F'(a_0)(a-a_0)+[-M_2r^2/2,M_2r^2/2]$ for $a\in I$, with
$M_2\ge\sup_I|F''|$ and $r:=\max\{a_0-\underline a,\overline a-a_0\}=\sup_{a\in I}|a-a_0|$;
the radius is the largest distance from $a_0$ to the range, not the half-width of
$I$, unless $a_0$ is its midpoint.

The state is enclosed from $u_\theta=BN_\theta$ by product rules,
$\partial_tu_\theta=B\,\partial_tN_\theta$ and $\gradx u_\theta=N_\theta\gradx B+B\gradx N_\theta$;
the flux head is unmasked; $\lambda_\psi$ is enclosed by interval squaring of $M_\psi$;
closed-form data are enclosed directly on each cell. This yields intervals for the
two branches \eqref{eq:branches}, for $\gamma=g-u_\theta$ and for $\lambda_\psi$. On
trace cells, $u_0-\widehat v(0)$ is enclosed by the hull of $u_0-u_\theta(0,\cdot)$ and
$u_0-g(0,\cdot)$ unless the sign of $\gamma(0,\cdot)$ is resolved. Trace cells
contribute to the same sums and are enclosed even when only the interior partition
is refined. A smooth Taylor remainder must not be used across a data interface;
there, the hull of the branch enclosures is used.

For the comparisons of \Cref{sec:refinement}: an affine form
$a=a_0+\sum_ja_j\varepsilon_j$, $|\varepsilon_j|\le1$, has range $[a_0-r,a_0+r]$,
$r=\sum_j|a_j|$, and a product appends one fresh symbol of radius $r_ar_b$;
second-order Taylor models use the cubic Lagrange remainder
$\tfrac16\sum_{i,j,k}M_{ijk}r_ir_jr_k$ with validated third-derivative ranges.
Affine, interval and Taylor ranges are intersected only when each is already valid;
an empty intersection is an error, not a certificate.

\subsection{Soundness basis of the finite-precision evaluation}\label{app:soundness}
A reported upper sum is valid if every per-cell bound follows the rules above and
every arithmetic step returns an interval containing its exact result.
\Cref{tab:soundness} lists the argument for the second property and its
assumptions.

\begin{table}[htbp]
\centering
\caption{Primitives of the finite evaluation, the argument for their outward
enclosure, and its assumptions.}
\label{tab:soundness}
\footnotesize
\begin{tabular}{p{.21\linewidth}p{.49\linewidth}p{.22\linewidth}}
\toprule
Primitive & Enclosure argument & Assumptions\\
\midrule
$+,-,\times,\div,\sqrt{\ }$ (binary64) & Correctly rounded IEEE~754 operations; each endpoint moved one unit in the last place outward after every operation. & Round to nearest; no fused multiply--add or extended intermediates; no flush to zero.\\
Sums of products (TM1 path) & Sequences of the operations above, each rounded outward, so the result encloses the exact sum in any order; no BLAS. & As above.\\
$\tanh$ and derivatives & Monotonicity; endpoint values from an Arb-based dyadic lookup with Taylor remainder bounds, converted outward; derivative ranges from $P_q(\tanh)$; no system transcendental as endpoint oracle. & Ball semantics of Arb \cite{Johansson2017}.\\
Taylor-model remainders & Lagrange remainders with second-derivative bounds from the previous row on the enclosed pre-activation range $I$, with radius $\sup_{a\in I}|a-a_0|$; the expansion point $a_0$ is checked to lie in $I$. & Rows above.\\
Data $f,g,u_0$ and derivatives & Arb ball evaluation on the cell at $\ge64$ bits, converted outward. & Ball semantics of Arb.\\
Volumes, reductions, constants & Exact dyadic volumes; endpoints converted exactly to binary rationals; Arb ball sums; $\overline c_F^{\,2}$ an upward-rounded rational; $N$ a $192$-bit lower endpoint. & Ball semantics of Arb.\\
Coverage & Interior and trace partitions are checked to cover $\overline Q$ and $\overline\Omega$; a NaN, overflow or empty intersection is an error. & None.\\
Lower sums for $N_{\rm lo}$ & Lower endpoints of squared TM1 ranges ($0$ if the range contains $0$), rounded towards $-\infty$; $N_{\rm lo}$ in $192$-bit Arb rounded down. & As above.\\
Affine comparison only & BLAS products enlarged by the a~priori bound $\gamma_{k+2}\sum|w_{ij}||z_j|$ \cite{Higham2002} plus an underflow term. & Standard floating-point model.\\
\bottomrule
\end{tabular}
\end{table}

The first two rows are consequences of IEEE~754 in the stated environment; the
others rest on the documented semantics of Arb balls. The released code identifies, for each row, where the rule is implemented, and an
environment check executed in every process (NumPy 1.26.4, OpenBLAS 0.3.23, Ubuntu
22.04) asserts round-to-nearest-even arithmetic, no flush to zero and no fused
multiply--add in the operations used. Two checks that share no code with the TM1
evaluation are applied to seven
certificates (long-disk seeds $0$--$2$ and ellipse seed $0$ at $8\cdot10^6$ cells,
Stefan seeds $0$--$2$ at $512^2$ cells), on random, largest-contribution, trace and
adversarial cells, $67{,}200$ (cell, term) checks in all. (i) Each term is evaluated
by networks built in $192$-bit Arb at $33$ points per space--time direction of every
cell, with binary64 screening and re-evaluation of corners and extremes in Arb; no
value exceeds the reported upper endpoint (largest ratio $0.99999$). (ii) The hull of a naive Arb interval extension over $512$ subcells is a valid
enclosure computed by a different algorithm. Two valid upper bounds need not be
ordered, so this is a consistency comparison rather than a test of validity; in all
checks the TM1 endpoint is the smaller one (largest ratio of TM1 to bisection
endpoint $0.92$). Check (i) is a necessary condition for enclosure validity. Neither
check establishes it: validity rests on the enclosure rules of
\Cref{app:enclosure-rules} and their outward-rounded execution described above.

\subsection{Closed-form reference norms}\label{app:norms}
For Stefan, \eqref{eq:stefan-u} gives $u_x=H(t)+\mathbf1_{\{x<t\}}(1-e^{t-x})$ and
$\int_0^1u_x^2\dd x=-\tfrac12e^{2t}+2te^t-t^2-t+\tfrac12$, whose integral over
$(0,1)$ is $23/12-e^2/4$; since $\sigma=-u_x$, $\norm{\sigma}{\LQ}=N$.

For the disk, $\gradx u=\gradx g+2\kappa_0\zeta\gradx\zeta\,\mathbf1_{\{r<R(t)\}}$ with
$\gradx\zeta=-(x-c_0)/R(t)$. The integral of $|\gradx g|^2$ is elementary; in polar
coordinates about $c_0$ the remaining integrand is a polynomial in $r$, $\cos\varphi$,
$\sin\varphi$ with coefficients polynomial in $H(t)$ and $R(t)$, and symbolic
integration gives
\begin{equation}\label{eq:disk-reference}
\begin{aligned}
N^2&=\frac{1408}{135}+\frac{128e^2}{45}-\frac{512e}{45}
+\frac{17878909\,\pi}{80640000}-\frac{833879\,e\pi}{9600000}\\
&\simeq0.474132,\qquad N\simeq0.688572,
\end{aligned}
\end{equation}
checked against two independent tensor Gauss--Legendre quadratures to $15$ digits.

For the ellipse, $\gradx u=\gradx u_b+2k_0\zeta_-\gradx\zeta$; in the coordinates
$x_1=0.43+a(t)\rho\cos\varphi$, $x_2=0.57+d(t)\rho\sin\varphi$ the integrand over the
moving ellipse is a polynomial in $\rho$, $\cos\varphi$, $\sin\varphi$, and
\begin{equation}\label{eq:ellipse-reference}
\begin{aligned}
N^2&=\frac{2356}{23625}- \frac{341875988641 \pi}{2625000000000000} - \frac{\pi \ln2}{61250}
- \frac{\pi \ln11}{375000} + \frac{\pi \ln6}{375000} + \frac{\pi \ln15}{183750},\\
N&\simeq0.3151529485,
\end{aligned}
\end{equation}
checked against two independent quadratures to $27$ digits. In all three cases the
validated bracket is the $192$-bit ball evaluation of the closed form; the Friedrichs
constants $1/\pi^2$ and $1/(2\pi^2)$ are enclosed the same way and rounded upward.

\FloatBarrier

\bibliographystyle{elsarticle-num-names}
\bibliography{refs}

\end{document}